\documentclass[preprint,12pt]{elsarticle}
\usepackage{amsmath,amsthm,amsfonts,amssymb}
\usepackage{mathtools}

 \usepackage[active,new,old]{correct} 

\usepackage{graphicx,latexsym}
 \usepackage{hyperref}
\usepackage{lscape, marvosym}

\newtheorem{thm}{Theorem}[section]
 \newtheorem{cor}[thm]{Corollary}
 \newtheorem{lem}[thm]{Lemma}
 \newtheorem{prop}[thm]{Proposition}
\newtheorem{ex}[thm]{Example}
 \newtheorem{defn}[thm]{Definition}

\newdefinition{remark}{Remark}

\newtheorem{discu}{Discussion:}

\numberwithin{equation}{section}
\newtheorem{conje}{Conjecture:}

\newcommand{\mc}[1]{\mathcal {#1}}
 \newcommand{\dg}{{\dagger}}
 \newcommand{\n}{{*_N}}

\journal{Linear and Multilinear Algebra}

\begin{document}

\begin{frontmatter}

\title{\sc{Further results on generalized inverses of tensors via  the Einstein
product}}

\vspace{-.4cm}

\author{Ratikanta Behera$^*$ and Debasisha Mishra$^\dag$}

\address{$^{*}$ Jean Kuntzmann Laboratory,\\
                Joseph Fourier University,\\
                51 rue des mathematiques, \\
                38041 Grenoble cedex 09, France.\\
                email: ratikanta.behera\symbol{'100}imag.fr.\\

                \vspace{.3cm}

                       $^{\dag}$ Department of Mathematics,\\
                        National Institute of Technology Raipur, India.
                        \\email: dmishra\symbol{'100}nitrr.ac.in. }

\begin{abstract}
The notion of the Moore-Penrose inverse of tensors with the Einstein
product was introduced, very recently.  In this paper, we further
elaborate this theory by producing a few characterizations of
different generalized inverses of tensors. A new method to compute
the Moore-Penrose inverse of tensors is  proposed. Reverse order
laws for several generalized inverses of tensors are also presented.
In addition to these, we discuss general solutions of multilinear
systems of
tensors using such theory.\\
\end{abstract}

\vspace{1cm}

\begin{keyword}
Moore-Penrose inverse; Tensor; Matrix; Einstein product.

\vspace{.3cm} { \it AMS subject classifications:} 15A69, 15A09
\end{keyword}

\end{frontmatter}

\newpage
\section{Introduction}\label{intro}

\subsection{Background and motivation}
The Moore-Penrose inverse of an arbitrary matrix (including singular and rectangular) has
many applications in statistics, prediction theory, control system analysis,  curve fitting,
 numerical analysis and solution of linear integral equations. But increasing ability of data
 collection systems to store huge volumes of multidimensional data and the recognition of potential
 modeling accuracy, matrix representation of data analysis is not enough to represent all the information
 content of the multiway data in different fields, including computer image and human motion recognition
 \cite{BraliNT13}, signal processing \cite{SidBrG00,  SidGiB00}, and many other areas using multiway
 data analysis \cite{CopBol89, Sho13}.  There has been a recent surge in the research and utilization
  of tensors (see the articles \cite{TamBr09, Qil05, SteCh12}) which are high-order generalization of matrices.

Tensor models are employed in numerous disciplines addressing the
problem of finding the multilinear structure in multiway data sets.
 Multilinear systems model many phenomena in engineering and science.
  For example, in continuum physics and engineering, isotropic and anisotropic elasticity
  are modelled \cite{LaiRu09} as multilinear systems. Further, the reverse order laws for
  generalized inverses of matrices \cite{BenGr03} play an important role in  theoretic research
   and numerical computations in ill-posed problems, optimization problems, and statistics problems.
    In addition, the reverse order laws for generalized inverses are also applied to
the generalized least squares problem and the weighted perturbation theory.
  It will be more applicable if we investigate reverse order laws for generalized inverses of
  tensors,
  and hence reverse order of generalized inverses of tensors will open different paths of the above areas.
   It is thus of interest to study theory
    of generalized inverses of a tensor via the Einstein product and its applications.
  In this direction, the Moore-Penrose inverse of a tensor was introduced in \cite{LizBaCY15}, very
     recently, and then the authors (\cite{LizBaCY15})  used the same notion to solve  multilinear systems of tensors.

In this paper, we further study  on generalized inverses of tensors.
This study can lead to the  enhancement of the computation of
generalized inverses
 of tensors along with solutions of multilinear structure in multidimensional systems.
In this regard, we  introduce different generalized inverses of
tensors and then provide different characterizations of generalized
inverses of tensors.
 We also   present reverse order laws for several
generalized inverses of tensors. Besides these,  a new method for
computing the Moore-Penrose inverse of a tensor is obtained.
Further, applications of these notions to multilinear systems are
studied, where we present a different approach to prove  Theorem
4.1, \cite{LizBaCY15}.

\subsection{Outline}
We organize the paper as follows. In the next subsection, we
introduce some notations and definitions which are helpful in
proving the main results in Section 2 and Section 3. In Section 2,
we prove several results concerning different generalized inverses.
In particular, we propose a method for computing the Moore-Penrose
inverse of a tensor. A few applications of these results to
multilinear systems of tensors are discussed in Section 3.

\subsection{Notations and definitions}
Multiway arrays, often referred to as tensors, are higher-order
generalizations of vectors and matrices. For a positive integer $N$,
let $[N] = \{ 1, \cdots, N\}$. An order $N$ tensor $\mc{A} =
(a_{{i_1}\cdots{i_N}})_{1\leq i_j \leq I_j} (j=1, \cdots,N)$ is a
multidimensional array with $I_1 I_2 \cdots I_N$ entries. Let
$\mathbb{C}^{I_1\times\cdots\times I_N}$ and
$\mathbb{R}^{I_1\times\cdots\times I_N}$ be the sets of the order
$N$ dimension $I_1 \times \cdots \times I_N$ tensors over the
complex field $\mathbb{C}$ and the real field $\mathbb{R}$,
respectively. For instance,
 $\mc{A} \in \mathbb{C}^{I_1\times\cdots\times I_N}$ is a multiway array with
 $N$-th order and $I_1, I_2, \cdots, I_N$ dimension in the first, second, $\cdots$ , $N$th way,
 respectively. Each entry of $\mc{A}$ is denoted by $a_{i_1...i_N}$.
 For $N =3$,  $\mc{A} \in \mathbb{C}^{I_1\times I_2 \times I_3}$ is a third order
 tensor, and $a_{i_1 i_2 i_3}$ denotes the entry of that tensor. Many results
  on tensors which have attracted great interest can be found in the
  articles \cite{TamBr09, SteCh12}.

For a tensor $\mc{A} = (a_{{i_1}...{i_N}{j_1}...{j_M}})
 \in \mathbb{C}^{I_1\times\cdots\times I_N \times J_1 \times\cdots\times J_M}$, let $\mc{B}
  = (b_{{i_1}...{i_M}{j_1}...{j_N}}) \in \mathbb{C}^{J_1\times\cdots\times J_M \times I_1 \times\cdots\times I_N}$
   be the conjugate transpose of $\mc{A}$,
   where $b_{{i_1}...{i_M}{j_1}...{j_N}} =\overline{a}_{{j_1}...{j_M}{i_1}...{i_N}}$.
The tensor $\mc{B}$ is denoted by $\mc{A}^*$.
 When $b_{{i_1}...{i_M}{j_1}...{j_N}} ={a}_{{j_1}...{j_M}{i_1}...{i_N}}$, $\mc{B}$
 is the {\it transpose} of $\mc{A}$, and is  denoted by $\mc{A}^T$.
  A tensor $\mc{D} =(d_{{i_1}...{i_N}{i_1}...{i_N}})
  \in \mathbb{C}^{I_1\times\cdots\times I_N \times I_1 \times\cdots\times I_N} $ is called a {\it diagonal
   tensor} if all its entries are zero except for $d_{{i_1}...{i_N}{i_1}...{i_N}}.$ In
   case of
     all the diagonal entries $d_{{i_1}...{i_N}{i_1}...{i_N}} = 1$, we call $\mc{D}$ as a {\it  unit
     tensor}, and is denoted by $\mc{I}$. Similarly,  $\mc{O}$ denotes the zero tensor
      in case of all the entries are zero.  A
few more notations and definitions  are introduced below for
defining generalized inverses of a tensor.

For $\mc{A} \in \mathbb{C}^{I_1\times\cdots\times I_N \times K_1
\times\cdots\times K_N }$ and $\mc{B} \in
\mathbb{C}^{K_1\times\cdots\times K_N \times J_1 \times\cdots\times
J_M }$, the Einstein product \cite{Ein} of tensors $\mc{A}$ and
$\mc{B}$
 is defined by the operation $*_N$ via
\begin{equation}
(\mc{A}*_N\mc{B})_{i_1...i_Nj_1...j_M}
=\displaystyle\sum_{k_1...k_N}a_{{i_1...i_N}{k_1...k_N}}b_{{k_1...k_N}{j_1...j_M}}.
\end{equation}
Note that  $ \mc{A}*_N\mc{B} \in \mathbb{C}^{I_1\times\cdots\times
I_N \times J_1 \times\cdots\times J_M }$.
 The associative law of this tensor product holds. In the above formula, when
  $\mc{B} \in \mathbb{C}^{K_1\times\cdots\times K_N}$, then
\begin{equation}
(\mc{A}*_N\mc{B})_{i_1...i_N} = \displaystyle\sum_{_{k_1...k_N}}
a_{{i_1...i_N}{k_1...k_N}}b_{{k_1...k_N}},
\end{equation}
where $\mc{A}*_N\mc{B} \in \mathbb{C}^{I_1\times\cdots\times I_N}$.
On the other hand, when
 $\mc{A} \in \mathbb{C}^{I_1\times\cdots\times I_N}$ and $\mc{B}$ is a vector
  $b=(b_i)\in \mathbb{C}^{I_N}$, the product is defined by the operation $\times_N$ via
\begin{equation}
(\mc{A}\times_N b)_{i_1...i_{N-1}}
=\displaystyle\sum_{i_N}a_{i_1...i_N}b_{i_N},
\end{equation}
where the tensor $\mc{A}\times_N b \in
\mathbb{C}^{I_1\times\cdots\times I_{N-1}}.$  We next collect some
definitions and results which are essential to prove our main
results.

\begin{defn}{\cite{LizBaCY15}}
For a tensor $\mc{A} = (a_{i_1... i_Nj_1...j_M})
 \in \mathbb{C}^{I_1\times\cdots\times I_N \times J_1 \times\cdots\times
 J_M},\\
  \mc{A}_{(i_1...i_N|:)}= (a_{i_1...i_N:...:})\in \mathbb{C}^{J_1\times\cdots\times J_M}$ is a
  subblock of $\mc{A}$. $Vec(\mc{A})$ is obtained by lining up all the subtensors
  in a column, and  $t$-th subblock of $Vec(\mc{A})$ is  $\mc{A}_{(i_1...i_N|:)}$,
 where $t=i_N + \displaystyle\sum_{K=1}^{N-1} \left[ (i_K - 1) \displaystyle\prod_{L=K+1}^{N} I_L \right]$.
\end{defn}

\begin{defn}{\cite{LizBaCY15}}
The Kronecker product of $\mc{A} \in
\mathbb{C}^{I_1\times\cdots\times I_N \times J_1 \times\cdots\times
J_N}$
 and $\mc{B} \in \mathbb{C}^{K_1\times\cdots\times K_M \times L_1 \times\cdots\times L_M}$, denoted by
 $\mc{A}\otimes\mc{B} = (a_{i_1...i_Nj_1...j_N}\mc{B}),$ is a `Kr-block tensor' whose $(t_1, t_2)$
  subblock is $(a_{i_1...i_Nj_1...j_N}\mc{B})$ obtained via multiplying all the entries of $\mc{B}$ by
  a constant $a_{i_1...i_Nj_1...j_N},$ where\\
    $t_1=i_N + \displaystyle\sum_{K=1}^{N-1} \left[ (i_K - 1) \displaystyle\prod_{L=K+1}^{N} I_L \right]$ and
   $t_2=j_N + \displaystyle\sum_{K=1}^{N-1} \left[ (j_K - 1) \displaystyle\prod_{L=K+1}^{N} J_L
   \right]$.
\end{defn}

\begin{remark}
From the above definition, it can be immediately seen that the
Kronecker product is not commutative, i.e., $\mc{A}\otimes\mc{B}
\neq \mc{B}\otimes\mc{A}$.
\end{remark}

The next result can be verified easily.

\begin{prop}\label{Krop}
Let $\mc{A} \in \mathbb{C}^{I_1\times\cdots\times I_N \times J_1
\times\cdots\times J_N}$,
  $\mc{B} \in \mathbb{C}^{K_1\times\cdots\times K_M \times L_1 \times\cdots\times L_M}$,
  $\mc{C} \in \mathbb{C}^{K_1\times\cdots\times K_M \times L_1 \times\cdots\times L_M}$ and
  $\mc{D} \in \mathbb{C}^{J_1\times\cdots\times J_N \times L_1 \times\cdots\times L_M} $. Then

\item[(a)] $(\mc{A}\otimes \mc{B})^* = \mc{A}^* \otimes \mc{B}^*$;
\item[(b)] $\mc{A}\otimes(\mc{B}\otimes\mc{C})
 = (\mc{A}\otimes\mc{B})\otimes \mc{C}$;
\item[(c)] $\mc{A}\otimes(\mc{B} + \mc{C})
 = \mc{A}\otimes\mc{B}+ \mc{A}\otimes \mc{C}$ and
 $(\mc{B} + \mc{C})\otimes\mc{A}= \mc{B}\otimes\mc{A}+ \mc{C}\otimes{A}$;
\item[(d)] $(\mc{A}\otimes\mc{B})*_MVec(\mc{D}) = Vec{(\mc{A}*_N\mc{D}*_M\mc{B}^T)}.$
\end{prop}

We have another result presented below on Kronecker product of
tensors.

\begin{lem}\label{kp} ({Proposition 2.3, \cite{LizBaCY15}})\\
Let $\mc{A} \in \mathbb{C}^{I_1\times\cdots\times I_N \times J_1
\times\cdots\times J_N}$,
  $\mc{B} \in \mathbb{C}^{K_1\times\cdots\times K_M \times L_1 \times\cdots\times L_M}$,
   $\mc{C} \in \mathbb{C}^{J_1\times\cdots\times J_N \times I_1 \times\cdots\times I_N}$
     and $\mc{D} \in \mathbb{C}^{L_1\times\cdots\times L_M \times K_1 \times\cdots\times K_M} $. Then
\center{$(\mc{A}\otimes\mc{B})*_M(\mc{C}\otimes\mc{D})=(\mc{A}*_N\mc{C})\otimes(\mc{B}*_M\mc{D})$}.
\end{lem}

\section{Generalized inverses of tensors}
This section begins with the definitions of various generalized
inverses of a tensor  $\mc{A} \in \mathbb{C}^{I_1\times\cdots\times
I_N \times J_1 \times\cdots\times J_N}$ via
  the Einstein product. We first recall the definition of the Moore-Penrose inverse
of a tensor which was introduced in \cite{LizBaCY15}, very recently.


\begin{defn}(Definition 2.2, \cite{LizBaCY15})\label{defmpi}
Let $\mc{A} \in \mathbb{C}^{I_1\times\cdots\times I_N \times J_1
\times ... \times J_N}$. The tensor $\mc{X} \in
\mathbb{C}^{J_1\times\cdots\times J_N \times I_1 \times\cdots\times
I_N}$ satisfying the following four tensor equations:
\begin{enumerate}
\item[(1)] $\mc{A}*_N\mc{X}*_N\mc{A}= \mc{A};$
\item[(2)] $\mc{X}*_N\mc{A}*_N\mc{X}= \mc{X};$
\item[(3)] $(\mc{A}*_N\mc{X})^*=\mc{A}*_N\mc{X};$
\item[(4)] $(\mc{X}*_N\mc{A})^* = \mc{X}*_N\mc{A},$
\end{enumerate}
is called  the \textbf{Moore-Penrose inverse} of $\mc{A}$, and is
denoted by $\mc{A}^{\dg}$.
\end{defn}

Let $\lambda$ be a nonempty subset of $\{1,2,3,4\}$. Then a tensor
$\mc{X} \in \mathbb{C}^{J_1\times\cdots\times J_N \times I_1 \times
... \times I_N}$ is called a {\it $\{\lambda \}$-inverse} of
$\mc{A}$ if $\mc{X}$ satisfies equation $(i)$ for each $i\in
\lambda$. We denote such an inverse by $\mc{A}^{(\lambda)}$ while
$\mc{A}\{\lambda\}$ stands for the class of all $\{\lambda
\}$-inverses of $\mc{A}$. For $\lambda=\{1\}$ and $\lambda=\{1,
2\}$, $\mc{X}$ is called as  a {\it  generalized inverse} and  a
{\it reflexive generalized inverse} of $\mc{A}$, respectively.  If
$\lambda=\{1, 3\}$, then we have {\it $\{1, 3\}$-inverse} of
$\mc{A}$, and for $\lambda=\{1, 4\}$, we get {\it $\{1,
4\}$-inverse} of $\mc{A}$.

Brazell et al. \cite{BraliNT13} introduced the notion of ordinary
tensor inverse, and is as follows. A tensor $\mc{X} \in
\mathbb{C}^{I_1\times\cdots\times I_N \times I_1 \times\cdots\times
I_N}$ is called  the  {\it inverse} of $\mc{A}\in
\mathbb{C}^{I_1\times\cdots\times I_N \times I_1 \times\cdots\times
I_N}$  if it satisfies $\mc{A}*_N\mc{X}=\mc{X}*_N\mc{A}=\mc{I}$. It
is denoted  by $\mc{A}^{-1}$.
 A tensor $\mc{A}\in
\mathbb{C}^{I_1\times\cdots\times I_N \times I_1 \times\cdots\times
I_N}$ is {\it hermitian}  if  $\mc{A}=\mc{A}^*$. Further, a tensor
$\mc{A}\in \mathbb{C}^{I_1\times\cdots\times I_N \times I_1
\times\cdots\times I_N}$  is {\it unitary}  if  $\mc{A}\n
\mc{A}^*=\mc{A}^*\n \mc{A}=\mc{I}$, and {\it idempotent}  if $\mc{A}
\n \mc{A}= \mc{A}.$ In case of tensors of real entries, hermitian
and unitary tensors are called {\it symmetric} and {\it orthogonal}
tensors. (See Definition 3.15 and  Definition 3.16 of
\cite{BraliNT13}, respectively.)

In case of an invertible tensor $\mc{A}$, $\mc{A}^{\dg} =
\mc{A}^{(\lambda)}=\mc{A}^{-1}$. Next, we collect some known
properties of $\mc{A}^{\dg}$ (see Proposition 3.3, \cite{LizBaCY15})
which will be frequently used in this paper:
$(\mc{A}^{\dg})^{\dg}=\mc{A}$ and
$(\mc{A}^{*})^{\dg}=(\mc{A}^{\dg})^{*}$. We now move to prove a
result which was also proposed in Proposition 3.3 (4),
\cite{LizBaCY15}, and to do this we need the following lemma which
we call as singular value decomposition (SVD) of a tensor proved in
Theorem 3.17, \cite{BraliNT13} for a real tensor. The authors of
\cite{LizBaCY15} stated the same result for a complex tensor, and is
recalled next.

\begin{lem}{(Lemma 3.1, \cite{LizBaCY15})}
 A tensor $\mc{A} \in
\mathbb{C}^{I_1\times\cdots\times I_N \times J_1 \times\cdots\times
J_N}$
 can be decomposed  as $$\mc{A} = \mc{U}*_N\mc{B}*_N\mc{V}^*,$$
 where $\mc{U} \in \mathbb{C}^{I_1\times\cdots\times I_N \times I_1 \times\cdots\times I_N}$ and
 $\mc{V} \in  \mathbb{C}^{J_1\times\cdots\times J_N \times J_1 \times\cdots\times J_N}$ are unitary tensors,
 $\mc{B} \in \mathbb{C}^{I_1\times\cdots\times I_N \times J_1 \times\cdots\times J_N}$ is a
 tensor such that
 $(\mc{B})_{i_1...i_Nj_1...j_N} =0$, if $(i_1,\cdots,i_N) \neq (j_1,\cdots,j_N).$
 \end{lem}

 Existence and uniqueness of
$\mc{A}^{\dg} $ is shown in Theorem 3.2, \cite{LizBaCY15}. The
authors of \cite{LizBaCY15} also showed that $\mc{A}^{\dg} =
\mc{V}*_N\mc{B}^{\dg}*_N\mc{U}^*$ in the proof of Theorem 3.2,
\cite{LizBaCY15}. Using this, we will prove the condition (a) in the
next result.

\begin{lem}\label{revA}
Let $\mc{A}\in \mathbb{C}^{I_1\times\cdots\times I_N \times J_1
\times\cdots\times J_N }$. Then\\
 (a) $({\mc{A}^*} *_N \mc{A})^{\dg}=\mc{A}^{\dg}*_N ({\mc{A}^*})^{\dg}.$\\
 (b) $\mc{A}^{\dg} = ({\mc{A}^*}*_N{\mc{A})^{\dg}}*_N\mc{A}^*$.
 \end{lem}

\begin{proof}
(a) As an application of  SVD of $\mc{A}$, we have $\mc{A}^{\dg} =
\mc{V}*_N\mc{B}^{\dg}*_N\mc{U}^*$. Hence
\begin{eqnarray*}
({\mc{A}^*} *_N \mc{A})^{\dg} & = & (\mc{V}*_N\mc{B}^{*}*_N\mc{U}^*
*_N \mc{U}*_N\mc{B}*_N\mc{V}^*)^{\dg}\\
& = &  (\mc{V} *_N (\mc{B}^{*} *_N \mc{B})^{\dg} *_N\mc{V}^*\\
& = &  \mc{V} *_N \mc{B}^{\dg} *_N (\mc{B}^{*})^{\dg} *_N\mc{V}^*\\
& = & \mc{V} *_N \mc{B}^{\dg}*_N\mc{U}^* *_N \mc{U}*_N(\mc{B}^{*})^{\dg} *_N\mc{V}^* \\
& = & \mc{A}^{\dg}*_N{\mc{A}^*}^{\dg}.
\end{eqnarray*}

(b)  Let $\mc{C}=\mc{A}^* *_N \mc{A}$,
$\mc{B}=({\mc{A}^*}*_N{\mc{A})^{\dg}}*_N\mc{A}^*= \mc{C}^{\dg}
*_N\mc{A}^*$ and  $\mc{X}=\mc{A}$. Then $\mc{B} \n
\mc{X}=\mc{C}^{\dg} \n \mc{C}$. So $\mc{B} \n \mc{X}$ is hermitian.
Then it follows that  $\mc{B} \n \mc{X} \n \mc{B}=\mc{B}$ by using
condition (2) of Definition \ref{defmpi} for tensor $\mc{C}$.
Similarly, $\mc{X} \n \mc{B}= \mc{A} \n (\mc{A}^* \n\mc{A})^{\dg} \n
\mc{A}^*.$ So $\mc{X}$ satisfies condition (4)  of Definition
\ref{defmpi}. The remained condition (2) is shown hereunder.
\begin{eqnarray*}
 \mc{X} \n \mc{B} \n \mc{X}
& = &  \mc{A} \n (\mc{A}^* \n\mc{A})^{\dg} \n \mc{A}^*\n\mc{A}\\
& = &   \mc{A} \n \mc{A}^{\dg} \n (\mc{A}^*) ^{\dg} \n\mc{A}^*\n\mc{A}\\
& = & \mc{A} \n\mc{A}^{\dg} \n (\mc{A} \n\mc{A}^{\dg})^* \n\mc{A}\\
& = &  \mc{A} \n\mc{A}^{\dg}\n \mc{A} \n\mc{A}^{\dg}\n
\mc{A}=\mc{X}.
\end{eqnarray*}
\end{proof}

\begin{remark}
Nevertheless,  Lemma \ref{revA} (a) is not true if we replace
$\mc{A}^*$ by any other tensor $\mc{B}$, i.e., $(\mc{B}\n
\mc{A})^\dg \neq \mc{B}^\dg *_N \mc{A}^\dg$, where $\mc{A}$  and $
\mc{B} \in \mathbb{C}^{I_1\times\cdots\times I_N \times I_1
\times\cdots\times I_N}$.
\end{remark}

\begin{ex}
Let
 $\mc{A} = (a_{ijkl})_{1 \leq i,j,k,l \leq 2}  \in \mathbb{R}^{2\times 2\times 2\times 2}$ and $\mc{B}
 = (b_{ijkl})_{1 \leq i,j,k,l \leq 2}\in \mathbb{R}^{2\times 2\times 2\times 2}$ be two tensors such that
\begin{eqnarray*}
a_{ij11} =
    \begin{pmatrix}
    0 & 0\\
    0 & 1\\
    \end{pmatrix},
a_{ij21} =
    \begin{pmatrix}
    1 & -1\\
    0 & 0\\
    \end{pmatrix},
a_{ij12} =
    \begin{pmatrix}
    0 & 1\\
    0 & 0\\
    \end{pmatrix},
a_{ij22} =
    \begin{pmatrix}
    1 & 0\\
    -1 & 0\\
    \end{pmatrix}, ~
\end{eqnarray*}
and
\begin{eqnarray*}
b_{ij11} =
    \begin{pmatrix}
    1 & -1\\
    0 & 0\\
    \end{pmatrix},
b_{ij21} =
    \begin{pmatrix}
    0 & 1\\
    0 & 0\\
    \end{pmatrix},
b_{ij12} =
    \begin{pmatrix}
    0 & 0\\
    -1 & 0\\
    \end{pmatrix},
b_{ij22} =
    \begin{pmatrix}
    0 & 0\\
    1 & 0\\
    \end{pmatrix},
\end{eqnarray*}
respectively.  Then   ${\mc{A}}^{\dag} = (x_{ijkl})_{1 \leq i,j,k,l
\leq 2}  \in \mathbb{R}^{2\times 2\times 2\times 2}$
 and ${\mc{B}}^{\dag} = (y_{ijkl})_{1 \leq i,j,k,l \leq 2}
   \in \mathbb{R}^{2\times 2\times 2\times 2}$, where
\begin{eqnarray*}
x_{ij11} =
    \begin{pmatrix}
    0 & 1\\
    1 & 0\\
    \end{pmatrix},
x_{ij21} =
    \begin{pmatrix}
    0 & 1\\
    1 & -1\\
    \end{pmatrix},
x_{ij12} =
    \begin{pmatrix}
    0 & 1\\
    0 & 0\\
    \end{pmatrix},
x_{ij22} =
    \begin{pmatrix}
    1 & 0\\
    0 & 0\\
    \end{pmatrix}, ~
\end{eqnarray*} and
\begin{eqnarray*}
y_{ij11} =
    \begin{pmatrix}
    1 & 0\\
    1 & 0\\
    \end{pmatrix},
y_{ij21} =
    \begin{pmatrix}
    0 & -\frac{1}{2}\\
    0 & \frac{1}{2}\\
    \end{pmatrix},
y_{ij12} =
    \begin{pmatrix}
    0 & 0\\
    1 & 0\\
    \end{pmatrix},
y_{ij22} =
    \begin{pmatrix}
    0 & 0\\
    0 & 0\\
    \end{pmatrix},
\end{eqnarray*}
respectively. So  $\mc{B}^{\dg} *_N \mc{A}^{\dg}  =
 (d_{ijkl})_{1 \leq i,j,k,l \leq 2} \in \mathbb{R}^{2\times 2 \times 2\times 2}$,
 where
\begin{eqnarray*}
d_{ij11} =
    \begin{pmatrix}
    0 & -\frac{1}{2}\\
    1 & \frac{1}{2}\\
    \end{pmatrix},
d_{ij21} =
    \begin{pmatrix}
    0 & -\frac{1}{2}\\
    1 & \frac{1}{2}\\
    \end{pmatrix},
d_{ij12} =
    \begin{pmatrix}
    0 & 0\\
    1 & 0\\
    \end{pmatrix},
d_{ij22} =
    \begin{pmatrix}
    1 & 0\\
    1 & 0\\
    \end{pmatrix},
\end{eqnarray*}
 and  ${(\mc{A}*_N\mc{B})}^{\dag} = (c_{ijkl})_{1 \leq i,j,k,l \leq
2}  \in \mathbb{R}^{2\times 2\times 2\times 2}$, where
\begin{eqnarray*}
c_{ij11} =
    \begin{pmatrix}
    0 & -\frac{1}{2}\\
    1 & \frac{1}{2}\\
    \end{pmatrix},
c_{ij21} =
    \begin{pmatrix}
    0 & 0\\
    0 & 0\\
    \end{pmatrix},
c_{ij12} =
    \begin{pmatrix}
    0 & 0\\
    1 & 0\\
    \end{pmatrix},
c_{ij22} =
    \begin{pmatrix}
    1 & 0\\
    1 & 0\\
    \end{pmatrix}.
\end{eqnarray*}
Hence
\begin{eqnarray*}
 (\mc{A}*_N\mc{B})^{\dg} \neq \mc{B}^{\dg} *_N \mc{A}^{\dg}.
\end{eqnarray*}
\end{ex}

 Some sufficient conditions are obtained below for the equality case.
\begin{prop}
For $\mc{A} \in \mathbb{C}^{I_1\times\cdots\times I_N \times J_1
\times\cdots\times J_N} $
 and $\mc{B} \in \mathbb{C}^{ J_1
\times\cdots\times J_N \times I_1\times\cdots\times I_N }$,
$(\mc{A}*_N\mc{B})^\dg = \mc{B}^\dg
*_N\mc{A}^\dg$, if one of the following conditions holds.\\
 \noindent(a) $\mc{B} =\mc{A}^*$.\\
 (b) $\mc{B} =\mc{A}^\dg$.\\
 (c) $\mc{A}^* *_N\mc{A} = \mc{I}$.\\
 (d) $\mc{B}*_N\mc{B}^* = \mc{I}$.
\end{prop}

Using the method as in the proof of Lemma \ref{revA}, one can prove
the next lemma.

\begin{lem}\label{lemm11}
Let $\mc{A}\in \mathbb{C}^{I_1\times\cdots\times I_N \times J_1
\times\cdots\times J_N }$. Then the following are true.  \\
(a) $\mc{A}^{\dg} =
{\mc{A}^*}*_N(\mc{A}*_N\mc{A}^*)^{\dg}.$\\
(b) $\mc{A}^* =\mc{A}^**_N\mc{A}*_N\mc{A}^{\dg}
=\mc{A}^{\dg}*_N\mc{A}*_N\mc{A}^*$. \\
(c) $\mc{A} =\mc{A}*_N\mc{A}^* *_N(\mc{A}^*)^\dg =
(\mc{A}^*)^\dg *_N \mc{A}^* *_N \mc{A}$. \\
(d) $(\mc{U}*_N\mc{A}*_N\mc{V})^{\dg} =
 \mc{V}^**_N\mc{A}^{\dg}*_N\mc{U}^*,~ \textnormal{where}~ \mc{U}$  and $\mc{V}$ are unitary
 tensors.\\
(e) $\mc{A}^\dg*_N\mc{A} = \mc{A}^* *_N(\mc{A}^*)^\dg$ and
$\mc{A}*_N\mc{A}^\dg = (\mc{A}^*)^\dg *_N\mc{A}^*$.
 \end{lem}

Using Lemma \ref{lemm11} (b) and (c), one can prove the result
obtained below.

\begin{cor}
Let $\mc{A} \in \mathbb{C}^{I_1\times\cdots\times I_N \times I_1
\times\cdots\times I_N}$
 and $\mc{A}*_N\mc{A}^\dg = \mc{A}^\dg *_N\mc{A}$. Then
\begin{itemize}
\item[(a)] there exists a
$\mc{X} \in \mathbb{C}^{I_1\times\cdots\times I_N \times I_1
\times\cdots\times I_N}$ such that $\mc{A}*_N\mc{X} = \mc{A}^*$, and
\item[(b)] there exists a
 $\mc{Y} \in \mathbb{C}^{I_1\times\cdots\times I_N \times I_1 \times\cdots\times I_N}$
 such that $\mc{A}^* *_N\mc{Y} = \mc{A}$.
\end{itemize}
\end{cor}

Similarly, the following lemma can be proved.

\begin{lem}
Let $\mc{A}\in \mathbb{C}^{I_1\times\cdots\times I_N \times I_1
\times\cdots\times I_N }$. Then the following are true.  \\
(a) If $\mc{A}$ is hermitian and idempotent, then $\mc{A}^{\dg} =
\mc{A}$.\\
(b) $\mc{A}^{\dg} = \mc{A}^*$ if and only if $\mc{A}^* *_N \mc{A}$
is idempotent.\\
(c) $\mc{A} *_N \mc{A}^\dg$, $\mc{A}^\dg *_N\mc{A}$, $\mc{I}
-\mc{A}*_N\mc{A}^\dg$ and
 $\mc{I} - \mc{A}^\dg *_N\mc{A}$ are all idempotent.
\end{lem}

Analogous results to all the above discussed results for matrices
can be found in Chapter 1, \cite{BenGr03}. We now proceed to discuss
a few properties of $\{1\}$-inverse of a tensor, below.

\begin{thm}
The following three conditions are equivalent:\\
 (i) $(\mc{A}^{(1)})^*\{1\}= (\mc{A}^*)^{(1)}\{1\}$.\\
 (ii) $ \mc{A} \n (\mc{A}^*\n \mc{A})^{(1)}\n \mc{A}^*\n \mc{A} =\mc{A}$.\\
 (iii) $(\mc{A} \n (\mc{A}^*\n \mc{A})^{(1)}\n \mc{A}^*)^*=\mc{A} \n (\mc{A}^*\n \mc{A})^{(1)}\n \mc{A}^*.$
\end{thm}

\begin{proof}
(i) The condition $\mc{A} \n \mc{A}^{(1)}\n \mc{A}= \mc{A}$ implies
$\mc{A}^* \n (\mc{A}^{(1)})^*\n \mc{A}^*= \mc{A}^*$. Hence
$(\mc{A}\{1\})^*\subseteq  \mc{A}^*\{1\}$. The other implication
follows from $(\mc{A}^*\{1\})^*\subseteq  \mc{A}\{1\}$ because of
$\mc{A}^* \n (\mc{A}^*)^{(1)}\n \mc{A}^*= \mc{A}^*$.

(ii)  Let $\mc{R}=(\mc{I}- \mc{A} \n (\mc{A}^*\n \mc{A})^{(1)}\n
\mc{A}^*)\n  \mc{A}= \mc{A}*_N(\mc{I}-  (\mc{A}^*\n \mc{A})^{(1)}\n
\mc{A}^*\n  \mc{A}).$ Then
\begin{eqnarray*}
&& \mc{R}^* \n \mc{R}\\
&=& (\mc{I}-  (\mc{A}^*\n \mc{A})^{(1)}\n \mc{A}^*\n \mc{A})^*\n
(\mc{A}^*\n  \mc{A}- \mc{A}^*\n \mc{A} \n (\mc{A}^*\n
\mc{A})^{(1)}\n \mc{A}^*\n \mc{A})\\
&=& \mc{O}.
\end{eqnarray*}

We thus have   $\mc{R}= \mc{O}$ which implies the desired result.

(iii) Let $\mc{R}$ be a $\{1\}$-inverse of $\mc{A}^*\n \mc{A}.$
Then, by (i), we  have $\mc{R}^*$ as a $\{1\}$-inverse of
$\mc{A}^*\n \mc{A}$. Again, setting $\mc{S}=(\mc{R}+\mc{R}^*)/2$, we
get $\mc{S}$ as $\{1\}$-inverse of $\mc{A}^*\n \mc{A}$, and also
$\mc{S}$ is hermitian. Let  $\mc{G}= \mc{A}\n \mc{S}\n \mc{A}^*-
\mc{A} \n (\mc{A}^*\n \mc{A})^{(1)}\n \mc{A}^*.$ Then
\begin{eqnarray*}
&& \mc{G}^*\n \mc{G} \\
&=& (\mc{A}\n \mc{S}- \mc{A} \n (\mc{A}^*\n \mc{A})^{(1)})^*\n
(\mc{A}^*\n \mc{A}\n \mc{S}\n \mc{A}^* - \mc{A}^*\n \mc{A}\n (\mc{A}^*\n \mc{A})^{(1)}\n \mc{A}^*)\\
&=& \mc{O}.
\end{eqnarray*}
Hence $\mc{G}= \mc{O}$ which leads to (iii).

\end{proof}

The matrix version of the above result is obtained below for $A \in
\mathbb{C}^{m\times n }$.
\begin{cor}
The following three conditions are equivalent:\\
 (i) $ {A^{(1)}}^* \{1\}= (A^*)^{(1)}\{1\}$.\\
 (ii) $A(A^* A)^{(1)} A^* A =A$.\\
 (iii) $(A(A^* A)^{(1)} A^*)^*=A  (A^*A)^{(1)}A^*.$
\end{cor}

Next result collects sufficient conditions for
$(\mc{A}+\mc{B})^{(1)}$ to be a $\{ 1 \}$-inverse of $\mc{B}$.

\begin{thm}
Let $\mc{A}, ~ \mc{B} \in \mathbb{C}^{I_1\times\cdots\times I_N
\times J_1 \times\cdots\times J_N }$. If  $\mc{A}^{(1)}\n
\mc{B}=\mc{B}^{(1)}\n \mc{A}=\mc{O}$, then $(\mc{A}+\mc{B})^{(1)}$
is a $\{1\}$-inverse of $\mc{B}$.
\end{thm}

\begin{proof}
The condition $\mc{A}^{(1)}\n \mc{B}=\mc{O}$ yields
$\mc{B}=\mc{A}-\mc{A}\n \mc{A}^{(1)}\n \mc{A}+\mc{B}-\mc{A}\n
\mc{A}^{(1)}\n \mc{B}=(\mc{I}-\mc{A}\n \mc{A}^{(1)})\n
(\mc{A}+\mc{B})$. Using the other condition, one can  have
$\mc{B}=(\mc{A}+\mc{B})\n (\mc{I}-\mc{A}^{(1)}\n \mc{A})$. Then
\begin{eqnarray*}
 &&\mc{B}\n (\mc{A}+\mc{B})^{(1)} \n \mc{B}\\
 &=& (\mc{I}-\mc{A}\n
 \mc{A}^{(1)})\n (\mc{A}+\mc{B})\n (\mc{A}+\mc{B})^{(1)}\n
(\mc{A}+\mc{B})\n  (\mc{I}-\mc{A}^{(1)}\n \mc{A})\\
 &=& (\mc{I}-\mc{A}\n
\mc{A}^{(1)})\n (\mc{A}+\mc{B})\n (\mc{I}-\mc{A}^{(1)}\n
\mc{A})\\
&=&  (\mc{I}-\mc{A}\n \mc{A}^{(1)})\n \mc{B}\\
 &=& \mc{B}- \mc{A}\n
\mc{A}^{(1)} \n \mc{B}= \mc{B}.
\end{eqnarray*}
\end{proof}

The following one is obtained as  a corollary for matrices.
\begin{cor}
Let $A, ~ B \in \mathbb{C}^{m\times n }$. If  $A^{(1)}B=B^{(1)}A=O$, then $(A+B)^{(1)}$
is a $\{1\}$-inverse of $B$.
\end{cor}

Observe that $\mc{B}= \mc{A}\n \mc{A}^{(1)}\n \mc{B}$ implies
$\mc{B}= \mc{A}\n \mc{H}$ where $\mc{H}=\mc{A}^{(1)}\n \mc{B}$.
Conversely, if $\mc{B}= \mc{A}\n \mc{H}$, then premultiplying
$\mc{A}\n \mc{A}^{(1)}$, we have
 $$\mc{A}\n \mc{A}^{(1)}\n
\mc{B}=\mc{A}\n \mc{A}^{(1)}\n \mc{A} \n \mc{H}= \mc{A} \n
\mc{H}=\mc{B}.$$ Hence, in this case, $\{1\}$-inverse of a tensor
behaves like ordinary inverse of a tensor which is stated in the
next result.

\begin{thm}
Let $\mc{A},  \mc{B} \in \mathbb{C}^{I_1\times\cdots\times I_N
\times J_1 \times\cdots\times J_N }$. Then $\mc{B}= \mc{A}\n
\mc{A}^{(1)}\n \mc{B}$ if and only if $\mc{B}= \mc{A}\n \mc{H}$ for
some $\mc{H}.$ Similarly, $\mc{B}= \mc{B}\n \mc{A}^{(1)} \n \mc{A}$
if and only if $\mc{B}=  \mc{G}\n \mc{A}$ for some $\mc{G}.$
\end{thm}

An immediate consequence of the above result is shown next as a
corollary.

\begin{cor}
Let $\mc{A} \in \mathbb{C}^{I_1\times\cdots\times I_N \times J_1
\times\cdots\times J_N }$. Then\\
(i) $\mc{A}= \mc{A}\n (\mc{A^*}\n \mc{A}) ^{(1)}\n (\mc{A^*}\n \mc{A})$ and \\
(ii) $\mc{A^*}=  (\mc{A^*}\n \mc{A}) \n (\mc{A^*}\n \mc{A})^{(1)} \n
\mc{A^*}$.
\end{cor}

Another characterization of   $\{1\}$-inverse of a tensor is
presented below.

\begin{thm}
If $\mc{S}$ and $\mc{T}$ are two invertible tensors, and $\mc{G}$ is
a $\{1\}$-inverse of $\mc{A}$, then
$\mc{T}^{-1}*_N\mc{G}*_N\mc{S}^{-1}$ is a $\{1\}$-inverse of
 $\mc{B} = \mc{S}*_N\mc{A}*_N\mc{T}$. Moreover, every $\{1\}$-inverse of $\mc{B}$ is of this form.
\end{thm}

\begin{proof}
We have $\mc{T}^{-1}*_N\mc{G}*_N\mc{S}^{-1}\in \mc{B}\{1\}$ since
$\mc{B}*_N(\mc{T}^{-1}*_N\mc{G}*_N\mc{S}^{-1})*_N\mc{B}=
\mc{S}*_N\mc{A}*_N\mc{T}*_N(\mc{T}^{-1}*_N\mc{G}*_N\mc{S}^{-1})*_N\mc{S}*_N\mc{A}*_N\mc{T}
=\mc{S}*_N\mc{A}*_N\mc{G}*_N\mc{A}*_N\mc{T}= \mc{S}*_N\mc{A}*_N
\mc{T} = \mc{B}$. Again, let $\mc{K}$ be any $\{1\}$-inverse of
$\mc{B}$. We also have
$\mc{S}^{-1}*_N\mc{B}*_N\mc{T}^{-1}*_N(\mc{T}*_N\mc{K}*_N\mc{S})*_N\mc{S}^{-1}*_N\mc{B}*_N
\mc{T}^{-1} = \mc{S}^{-1}*_N\mc{B}*_N\mc{K}*_N\mc{B}*_N\mc{T}^{-1} =
\mc{S}^{-1}*_N\mc{B}*_N\mc{T}^{-1}$. But
$\mc{S}^{-1}*_N\mc{B}*_N\mc{T}^{-1} = \mc{A}$ as $\mc{B}
=\mc{S}*_N\mc{A}*_N\mc{T}$. Considering $\mc{G}
=\mc{S}*_N\mc{K}*_N\mc{T}$, we get
$\mc{A}*_N\mc{G}*_N\mc{A}=\mc{A}*_N(\mc{S}*_N\mc{K}*_N\mc{T})*_N\mc{A}=
\mc{A}*_N\mc{S}*_N\mc{S}^{-1}\mc{G}*_N\mc{T}^{-1}*_N\mc{A}=
\mc{A}*_N\mc{G}*_N\mc{A} = \mc{A}$, i.e., $\{1\}$-inverse of
$\mc{A}$, and $\mc{K} =\mc{T}^{-1}*_N\mc{G}*_N\mc{S}^{-1}$.
\end{proof}

Reverse order law for $\{1\}$-inverse of tensors is shown next.

\begin{thm}
$(\mc{A}\n \mc{B})^{(1)}= \mc{B}^{(1)} \n \mc{A}^{(1)}$ if and only
if $$(\mc{A}^{(1)}\n \mc{A}\n \mc{B} \n
\mc{B}^{(1)})^2=\mc{A}^{(1)}\n \mc{A}\n \mc{B} \n \mc{B}^{(1)}.$$
\end{thm}

\begin{proof}
Suppose that $(\mc{A}\n \mc{B})^{(1)}= \mc{B}^{(1)} \n
\mc{A}^{(1)}$. We then have $$\mc{A}\n \mc{B}\n \mc{B}^{(1)} \n
\mc{A}^{(1)}\n \mc{A}\n \mc{B}= \mc{A}\n \mc{B}.$$ Premultiplying
and postmultiplying both sides by $\mc{A}^{(1)}$ and $\mc{B}^{(1)}$,
respectively,  we get $$(\mc{A}^{(1)}\n \mc{A}\n \mc{B} \n
\mc{B}^{(1)})^2=\mc{A}^{(1)}\n \mc{A}\n \mc{B} \n \mc{B}^{(1)}.$$
\end{proof}
We next have a corollary to the above result for rectangular matrices of suitable order.

\begin{cor}
$(AB)^{(1)}= B^{(1)}A^{(1)}$ if and only if
$$(A^{(1)}ABB^{(1)})^2=A^{(1)}ABB^{(1)}.$$
\end{cor}

Let $\mc{A} = (a_{i_1\cdots i_N j_1 \cdots j_M}) \in
\mathbb{C}^{I_1\times\cdots\times I_N \times J_1 \times\cdots\times
J_M}$ and $\mc{B} = (b_{i_1\cdots i_N k_1 \cdots k_M}) \in
\mathbb{C}^{I_1\times\cdots\times I_N \times K_1 \times\cdots\times
K_M}$. Then {\it row block tensor} consisting  of $\mc{A}$ and
$\mc{B}$ is denoted by
\begin{equation}\label{eq223}
[\mc{A} ~ \mc{B}] \in \mathbb{C}^{\alpha^N\times\beta_1\times \cdots \times \beta_M},
\end{equation}
where $\alpha^N = I_1\times\cdots\times I_N, \beta_i = J_i + K_i, i
= 1, \cdots, M$, and is defined by
\begin{equation*}
[\mc{A} ~ \mc{B}]_{i_1 \cdots i_N l_1 \cdots l_M} =
\begin{cases}
a_{i_1 \cdots i_N l_1 \cdots l_M}, & i_1 \cdots i_N \in [I_1] \times \dots \times [I_N], l_1 \cdots l_M \in [J_1] \times \cdots \times [J_M];
\\
b_{i_1 \cdots i_N l_1 \cdots l_M}, & i_1 \cdots i_N \in [I_1] \times \dots \times [I_N], l_1 \cdots l_M \in \Gamma_1 \times \cdots \times \Gamma_M;
\\
0, & \textnormal{otherwise}.
\end{cases}
\end{equation*}
where $\Gamma_i = \{ J_i +1, \cdots, J_i+K_i\}, i=1,\cdots, M.$

Let $\mc{C} = (c_{j_1 \cdots j_M i_1 \cdots i_N}) \in
\mathbb{C}^{J_1\times\cdots\times J_M \times I_1 \times\cdots\times
I_N}$ and $\mc{D} = (d_{k_1 \cdots k_M i_1 \cdots i_N})  \in
\mathbb{C}^{K_1\times\cdots\times K_M \times I_1 \times\cdots\times
I_N}$. Then {\it column block tensor} consisting of $\mc{C}$ and
$\mc{D}$ is
\begin{equation}\label{eq224}
\left[%
\begin{array}{c}
  \mc{C} \\
  \mc{D} \\
\end{array}%
\right]= [\mc{C}^T ~ \mc{D}^T]^T \in \mathbb{C}^{\beta_1 \times
\cdots \times \beta_M\times\alpha^N}.
\end{equation}
For $\mc{A}_1 \in \mathbb{C}^{I_1\times\cdots\times I_N \times J_1
\times\cdots\times J_M}, \mc{B}_1 \in
\mathbb{C}^{I_1\times\cdots\times I_N \times K_1 \times\cdots\times
K_M}, \mc{A}_2 \in \mathbb{C}^{L_1\times\cdots\times L_N \times J_1
\times\cdots\times J_M}$ and $ \mc{B}_2 \in
\mathbb{C}^{L_1\times\cdots\times L_N \times K_1 \times\cdots\times
K_M}$, we denote $\tau_1 = [\mc{A}_1 ~ \mc{B}_1]$ and $\tau_2 =
[\mc{A}_2 ~ \mc{B}_2]$ as the {\it row block tensors}.
 The {\it column block tensor}  $ \left[
\begin{array}{c}
  {\tau}_1 \\
  {\tau}_2 \\
\end{array}
\right]
$ can be written as
\begin{equation}\label{eq225}
\left[
\begin{array}{c}
  \mc{A}_1 ~~ \mc{B}_1\\
  \mc{A}_2 ~~ \mc{B}_2\\
\end{array}
\right] \in \mathbb{C}^{\rho_1\times\cdots\times \rho_N \times \beta_1 \times\cdots\times \beta_M},
\end{equation}
where $\rho_i = I_i +L_i, i=1,\cdots,N; \beta_j = J_j + K_j$ and $j=1,\cdots , M.$

The product of some  block tensors is recalled next from
\cite{LizBaCY15}.

\begin{lem} (Proposition 2.4, \cite{LizBaCY15})\\
Let  $[\mc{A} ~ \mc{B}], \left[
\begin{array}{c}
  \mc{C}\\
  \mc{D}\\
\end{array}
\right]$ and $ \left[
\begin{array}{c}
  \mc{A}_1 ~~ \mc{B}_1\\
  \mc{A}_2 ~~ \mc{B}_2\\
\end{array}
\right]$ be in the form as in equations \eqref{eq223}, \eqref{eq224}
and \eqref{eq225}, respectively. Then
\begin{itemize}
\item[(a)] $\mc{F} *_N [\mc{A}~~\mc{B}] = [\mc{F}*_N \mc{A} ~~\mc{F}*_N \mc{B}]
 \in  \mathbb{C}^{\alpha^N\times\beta_1\times \cdots \times \beta_M}$;
\item[(b)] $\left[
\begin{array}{c}
  \mc{C}\\
  \mc{D}\\
\end{array}
\right] *_N \mc{F} = \left[
\begin{array}{c}
  \mc{C}*_N \mc{F}\\
  \mc{D}*_N \mc{F}\\
\end{array}
\right] \in  \mathbb{C}^{\beta_1\times \cdots \times\beta_M\times\alpha^N}; $
\item[(c)] $[\mc{A}~~\mc{B}]*_M \left[
\begin{array}{c}
  \mc{C}\\
  \mc{D}\\
\end{array}
\right] = \mc{A}*_M\mc{C} + \mc{B}*_M\mc{D} \in \mathbb{C}^{\alpha^N \times \alpha^N};$
\item[(d)] $\left[
\begin{array}{c}
  \mc{C}\\
  \mc{D}\\
\end{array}
\right] *_N [\mc{A}~~\mc{B}] = \left[
\begin{array}{c}
  \mc{C} *_N \mc{A}~~ \mc{C} *_N \mc{B}\\
  \mc{D}*_N \mc{A}~~ \mc{D}*_N \mc{B}\\
\end{array}
\right] \in  \mathbb{C}^{\beta_1\times \cdots \times\beta_M\times \beta_1 \times \cdots \beta_M};$
\item[(e)] $\left[
\begin{array}{c}
  \mc{A}_1 ~~  \mc{B}_1\\
  \mc{A}_2 ~~ \mc{B}_2\\
\end{array}
\right] *_M \left[
\begin{array}{c}
  \mc{C}\\
  \mc{D}\\
\end{array}
\right] = \left[
\begin{array}{c}
  \mc{A}_1*_M\mc{C} ~~  \mc{B}_1*_M\mc{D}\\
  \mc{A}_2*_M\mc{C} ~~ \mc{B}_2*_M\mc{D}\\
\end{array}
\right] \in  \mathbb{C}^{\rho_1\times \cdots \times\rho_N\times\alpha^N};$

\item[(f)] $ [\mc{G}~~\mc{H}]*_N \left[
\begin{array}{c}
  \mc{A}_1 ~~  \mc{B}_1\\
  \mc{A}_2 ~~ \mc{B}_2\\
\end{array}
\right]=  [\mc{G}*_N\mc{A}_1 + \mc{H}*_N \mc{A}_2 ~~ \mc{G}*_N\mc{B}_1 + \mc{H}*_N \mc{B}_2]
 \in \mathbb{C}^{S_1\times \cdots \times S_N\times \beta_1\times\cdots \times\beta_M};$
\end{itemize}
where $\mc{F} \in \mathbb{C}^{\alpha^N\times\alpha^N},~ \mc{G} \in
\mathbb{C}^{S_1\times\cdots\times S_N\times I_1\times\cdots\times I_N}$ and
 $\mc{H} \in \mathbb{C}^{S_1\times\cdots\times S_N\times L_1\times\cdots\times L_N}.$
\end{lem}

Our last result on $\{1\}$-inverse of tensors is presented below.

\begin{thm}
For tensors $\mc{A}$ and $\mc{B}$ of suitable order, the following
results hold.\\
(i) $[\mc{A}~\mc{B}]\n [\mc{A} ~ \mc{B}]^{(1)}\n \mc{A}= \mc{A}$.\\
(ii) $(\mc{A} \n  \mc{A}^*+ \mc{B}\n  \mc{B}^*)
\n (\mc{A} \n  \mc{A}^*+ \mc{B}\n  \mc{B}^*)^{(1)} \n \mc{A} = \mc{A},$
if $\mc{B}^{(1)}*_N\mc{A} = \mc{O}$ and $\mc{A}^{(1)} *_N \mc{B}   = \mc{O}$. \\
(iii) $(\mc{A}+\mc{U}\n \mc{V})^{(1)}= \mc{A}^{(1)}- \mc{A}^{(1)}\n
\mc{U}\n (\mc{I} + \mc{V}\n \mc{A}^{(1)}\n \mc{U})^{(1)}\n  \mc{V}\n
\mc{A}^{(1)}$ if $\mc{V}=\mc{V}\n \mc{A}^{(1)}\n \mc{A}$ and
$\mc{U}= \mc{A}\n \mc{A}^{(1)}\n \mc{U}$.
\end{thm}

\begin{proof}
(i) Writting  $\mc{A}$ as $[\mc{A}~\mc{B}]\n \left[%
\begin{array}{c}
  \mc{I} \\
  \mc{O} \\
\end{array}%
\right]$ in $[\mc{A} ~\mc{B}]\n [\mc{A}~\mc{B}]^{(1)}\n \mc{A}$, we
get the equality.

(ii)
Setting $[\mc{A}~\mc{B}]=\mc{R}$, we have $\mc{A} \n
\mc{A}^{(1)}+ \mc{B}\n \mc{B}^{(1)}= \mc{R} \n \mc{R}^{(1)}$. Then $\mc{A}=[\mc{A} ~\mc{B}]\n \left[%
\begin{array}{c}
  \mc{I} \\
  \mc{O} \\
\end{array}%
\right]= \mc{R}  \n \left[%
\begin{array}{c}
  \mc{I} \\
  \mc{O} \\
\end{array}%
\right].$ Hence $(\mc{A} \n  \mc{A}^*+ \mc{B}\n  \mc{B}^* \n (\mc{A}
\n  \mc{A}^{*}+ \mc{B}\n  \mc{B}^*)^{(1)} \n \mc{A} = \mc{R} \n
\mc{R}^* \n (\mc{R} \n \mc{R}^*)^{(1)} \n \mc{R}  \n \left[%
\begin{array}{c}
  \mc{I} \\
  \mc{O} \\
\end{array}%
\right] = \mc{R}  \n \left[%
\begin{array}{c}
  \mc{I} \\
  \mc{O} \\
\end{array}%
\right]= \mc{A}.$

(iii) We have
\begin{eqnarray*}
&& (\mc{A}+\mc{U}\n \mc{V})\n [\mc{A}^{(1)}\n \mc{U}\n (\mc{I} +
\mc{V}\n \mc{A}^{(1)}\n \mc{U})^{(1)}\n  \mc{V}\n \mc{A}^{(1)})]\n
(\mc{A}+\mc{U}\n \mc{V})\\
&=&(\mc{A} + \mc{U}*_N\mc{V})*_N((\mc{A}^{(1)}*_N\mc{U} + \mc{A}^{(1)}*_N\mc{U}*_N\mc{V}*_N\mc{A}^{(1)}*_N\mc{U})*_N\mc{V}*_N\mc{A}^{(1)})\\
&& ~ *_N(\mc{A}+\mc{U}*_N\mc{V})\\
%
&=& (\mc{A} + \mc{U}*_N\mc{V})*_N(\mc{A}^{(1)}*_N\mc{U}*_N\mc{V}*_N\mc{A}^{(1)} + \mc{A}^{(1)}*_N\mc{U}*_N\mc{V}*_N\mc{A}^{(1)}\\
&& ~  *_N\mc{U}*_N\mc{V}*_N\mc{A}^{(1)})*_N(\mc{A}+\mc{U}*_N\mc{V})\\
&=&(\mc{A} + \mc{U}*_N\mc{V})*_N(\mc{A}^{(1)}*_N\mc{U}*_N\mc{V} + \mc{A}^{(1)}*_N\mc{U}*_N\mc{V}*_N\mc{A}^{(1)}*_N\mc{U}*_N\mc{V} + \mc{A}^{(1)}\\
&& ~ *_N\mc{U}*_N\mc{V}*_N\mc{A}^{(1)}*_N\mc{U}*_N\mc{V} +\mc{A}^{(1)}*_N\mc{U}*_N\mc{V}*_N\mc{A}^{(1)}*_N\mc{U}*_N\mc{V}*_N\mc{A}^{(1)}*_N\mc{U}*_N\mc{V})\\
&=& \mc{U}*_N\mc{V} + \mc{U}*_N\mc{V}*_N\mc{A}^{(1)}*_N\mc{U}*_N\mc{V} + \mc{U}*_N\mc{V}*_N\mc{A}^{(1)}*_N\mc{U}*_N\mc{V} \\
&& ~ + \mc{U}*_N\mc{V}*_N\mc{A}^{(1)}*_N\mc{U}*_N\mc{V}*_N\mc{A}^{(1)}*_N\mc{U}*_N\mc{V}\\
&&~  + \mc{U}*_N\mc{V}*_N\mc{A}^{(1)}*_N\mc{U}*_N\mc{V} +\mc{U}*_N\mc{V}*_N\mc{A}^{(1)}*_N\mc{U}*_N\mc{V}*_N\mc{A}^{(1)}*_N\mc{U}*_N\mc{V}\\
&& ~ +\mc{U}*_N\mc{V}*_N\mc{A}^{(1)}*_N\mc{U}*_N\mc{V}*_N\mc{A}^{(1)}*_N\mc{U}*_N\mc{V}\\
&& ~
+\mc{U}*_N\mc{V}*_N\mc{A}^{(1)}*_N\mc{U}*_N\mc{V}*_N\mc{A}^{(1)}
*_N\mc{U}*_N\mc{V}*_N\mc{A}^{(1)}*_N\mc{U}*_N\mc{V}\\
&=& (\mc{U}+\mc{U}*_N\mc{V}*_N\mc{A}^{(1)}*_N\mc{U} + \mc{U}*_N\mc{V}*_N\mc{A}^{(1)}*_N\mc{U} + \mc{U}*_N\mc{V}*_N\mc{A}^{(1)}*_N\mc{U}\\
&& ~ *_N\mc{V}*_N\mc{A}^{(1)}*_N\mc{U})*_N(\mc{V}+\mc{V}*_N\mc{A}^{(1)}*_N\mc{U}*_N\mc{V})\\
&=& (\mc{U} +\mc{U} *_N\mc{V}*_N\mc{A}^{(1)}*_N\mc{U}) *_N(\mc{I} +
\mc{V}*_N\mc{A}^{(1)}*_N\mc{U})*_N(\mc{V} +
\mc{V}*_N\mc{A}^{(1)}*_N\mc{U}*_N\mc{V})\\
 &=&(\mc{A}\n \mc{A}^{(1)}\n \mc{U}+ \mc{U}\n \mc{V}\n
\mc{A}^{(1)}\n \mc{U})\n  (\mc{I}+\mc{V}\n \mc{A}^{(1)}\n
\mc{U})^{(1)} \n (\mc{V}\n \mc{A}^{(1)}\n \mc{A}+\\
 & & \mc{V}\n
\mc{A}^{(1)}\n \mc{U}\n \mc{V}).
\end{eqnarray*}
Therefore
\begin{eqnarray*} & &(\mc{A}+\mc{U}\n \mc{V}) \n
(\mc{A}^{(1)}- \mc{A}^{(1)}\n \mc{U}\n (\mc{I} + \mc{V}\n
\mc{A}^{(1)}\n
\mc{U})^{(1)}\n  \mc{V}\n \mc{A}^{(1)})\n (\mc{A}+\mc{U}\n \mc{V})\\
\end{eqnarray*}

\begin{eqnarray*}
 &= & (\mc{A}+\mc{U}\n \mc{V}) \n \mc{A}^{(1)}\n (\mc{A}+\mc{U}\n
\mc{V}) -  (\mc{A}\n \mc{A}^{(1)}\n \mc{U}+ \mc{U}\n \mc{V}\n
\mc{A}^{(1)}\n \mc{U})\n \\
 & & (\mc{I}+\mc{V}\n \mc{A}^{(1)}\n
\mc{U})^{(1)} \n (\mc{V}\n \mc{A}^{(1)}\n \mc{A}+\mc{V}\n
\mc{A}^{(1)}\n \mc{U}\n \mc{V})\\
  &= & \mc{A}+2 \mc{U}\n \mc{V}+
\mc{U}\n \mc{V} \n \mc{A}^{(1)} \n \mc{U}\n \mc{V}-
 \mc{V}*_N (\mc{I}+\mc{V}\n \mc{A}^{(1)}\n \mc{U})\n  \mc{U}\\
  &= & \mc{A}+\mc{U}\n \mc{V}.
 \end{eqnarray*}
\end{proof}

 The following corollary is obtained in case of rectangular matrices.

\begin{cor}
For matrices $A$ and $B$ of suitable order, the following
results hold.\\
(i) $[A ~B] [A ~ B]^{(1)}A= A$.\\
(ii) $(AA^*+ BB^*)(AA^*+ BB^*)^{(1)}A= A$ if $B^{(1)}A = O$ and $A^{(1)} B =O$. \\
(iii) $(A+UV)^{(1)}= A^{(1)}- A^{(1)}U (I + VA^{(1)}U^{(1)})V
A^{(1)}$ if $V=VA^{(1)}A$ and $U= AA^{(1)}U$.
\end{cor}

Next result is for reflexive generalized inverse of a tensor
$\mc{A}\in \mathbb{C}^{I_1\times\cdots\times I_N \times J_1
\times\cdots\times J_N }$.

\begin{lem}\label{12}
Let  $\mc{Y}, \mc{Z} \in  \mc{A}{\{1\}}$ and $\mc{X} = \mc{Y}*_N\mc{A}*_N\mc{Z}$.
Then $\mc{X} \in  \mc{A}{\{1, 2\}}$.
\end{lem}

\begin{proof}
To have this, we have to  show that $\mc{X}$ satisfies conditions
(1) and (2) of Definition \ref{defmpi}. So,
\begin{eqnarray*}
\mc{A}*_N\mc{X}*_N \mc{A}
 &=& \mc{A}*_N(\mc{Y}*_N \mc{A}*_N\mc{Z})*_N\mc{A}\\
 &=& (\mc{A}*_N\mc{Y}*_N \mc{A})*_N\mc{Z}*_N\mc{A}\\
&=& \mc{A}*_N\mc{Z}*_N \mc{A}
= \mc{A}.\\
\textnormal{Again, we have } & & \\
\mc{X}*_N\mc{A}*_N \mc{X} &=& (\mc{Y}*_N\mc{A}*_N \mc{Z})*_N\mc{A}*_N(\mc{Y}*_N\mc{A}*_N\mc{Z})\nonumber\\
&=&  \mc{Y}*_N(\mc{A}*_N \mc{Z}*_N\mc{A})*_N\mc{Y}*_N\mc{A}*_N\mc{Z}\nonumber\\
&=&  \mc{Y}*_N(\mc{A}*_N \mc{Y}*_N\mc{A})*_N\mc{Z}\\
&=& \mc{Y}*_N\mc{A}*_N \mc{Z}
= \mc{X}.
\end{eqnarray*}
\end{proof}

A characterization of class of $\{1, 3\}$-inverse  of $\mc{A}\in
\mathbb{C}^{I_1\times\cdots\times I_N \times J_1 \times\cdots\times
J_N }$ is produced below.

\begin{thm}\label{13th}
The set $\mc{A}{\{ 1, 3\}}$ consists of all solutions $\mc{X}$ of
$$\mc{A}*_N\mc{X} = \mc{A}*_N\mc{A}^{(1,3)}.$$
\end{thm}

\begin{proof}
 Since $\mc{A}^{(1,3)}$ is a $\{1, 3\}$-inverse of $\mc{A}$, so
$\mc{A}*_N\mc{A}^{(1,3)} *_N\mc{A} = \mc{A}$ and $(\mc{A} *_N
\mc{A}^{(1,3)})^* = \mc{A} *_N \mc{A}^{(1,3)}.$ If $\mc{A}*_N\mc{X}
= \mc{A}*_N\mc{A}^{(1,3)}$, then we have
$$\mc{A}*_N\mc{X} *_N\mc{A} = \mc{A}*_N\mc{A}^{(1,3)} *_N\mc{A} = \mc{A}$$ and
$$(\mc{A} *_N \mc{X})^* = (\mc{A}*_N\mc{A}^{(1,3)})^* = \mc{A}*_N\mc{A}^{(1,3)} = \mc{A} *_N\mc{X}.$$
So $\mc{X}$ satisfies property (1) and (3) of Definition
\ref{defmpi}. Hence $\mc{X} \in \mc{A}{\{1,3\}}$. On the other hand,
suppose that $\mc{X} \in \mc{A}{\{1,3\}}.$ Then
\begin{eqnarray*}
\mc{A}*_N\mc{A}^{(1,3)} &=& \mc{A}*_N\mc{X}\n \mc{A}*_N\mc{A}^{(1,3)}\\
 &=& (\mc{A}*_N\mc{X})^**_N\mc{A}*_N\mc{A}^{(1,3)}\\
&=&  \mc{X}^* \n \mc{A}^*  = \mc{A}*_N\mc{X},~ \textnormal{as}~
(\mc{A}^{(1)})^* \in \mc{A}^*{\{ 1\}}.
\end{eqnarray*}
\end{proof}

Similarly, we have the following one for $\mc{A}\in
\mathbb{C}^{I_1\times\cdots\times I_N \times J_1 \times\cdots\times
J_N }$.

\begin{thm}\label{14th}
The set $\mc{A}{\{ 1,4\}}$ consists of all solutions $\mc{X}$ of
$$\mc{X}*_N\mc{A} = \mc{A}^{(1,4)}*_N\mc{A}.$$
\end{thm}

Here onwards, all our tensors in this section are assumed to be in
$\mathbb{C}^{I_1\times\cdots\times I_N \times J_1 \times\cdots\times
J_N }$ unless otherwise mentioned. Equivalent conditions for a
$\{1,4\}$-inverse is shown next.

\begin{thm}
The following three conditions are equivalent:\\
 (i) $\mc{B} \in \mc{A}\{1,4\}$.\\
 (ii) $\mc{B} \n \mc{A}\n \mc{A}^*=\mc{A}^*$.\\
 (iii) $\mc{B} \n \mc{A}= \mc{A}^*\n (\mc{A}\n \mc{A}^*)^{(1)}\n
 \mc{A}.$
\end{thm}

\begin{proof}
(i) $\Rightarrow $ (ii): From (i), we have $\mc{A} \n \mc{B}\n
\mc{A}=\mc{A}$  and $(\mc{B} \n \mc{A})^*=\mc{B} \n \mc{A}$. Hence
$(\mc{A} \n \mc{B}\n \mc{A})^*=\mc{A}^*$ yields $(\mc{B}\n
\mc{A})^*\n \mc{A}^*= \mc{A}^*$ which implies
 $\mc{B} \n \mc{A}\n \mc{A}^*=\mc{A}^*$.

 (ii) $\Rightarrow $ (iii): Postmultiplying (ii) by $(\mc{A} \n
 \mc{A}^*)^{(1)}\n \mc{A}$, we get (iii).

(iii) $\Rightarrow $ (i): Since $\mc{A}^* \n (\mc{A} \n
\mc{A}^*)^{(1)}\n \mc{A} \n \mc{A}^*=\mc{A}^*$, so we obtain\\
$\mc{A} \n \mc{A}^* \n (\mc{A} \n \mc{A}^*)^{(1)}\n \mc{A} =\mc{A}$
by taking conjugate transpose both sides. Hence $\mc{A}\n \mc{B}\n
\mc{A}=\mc{A}$. Next $(\mc{B}\n \mc{A})^*=(\mc{A}^* \n (\mc{A} \n
 \mc{A}^*)^{(1)}\n \mc{A})^*=\\
 \mc{A}^* \n (\mc{A} \n \mc{A}^*)^{(1)}\n \mc{A}=\mc{B}\n \mc{A}.$
\end{proof}

The matrix analogue is shown next for rectangular matrices of
suitable order.

\begin{cor}
The following three conditions are equivalent:\\
 (i) $B \in A\{1,4\}$.\\
 (ii) $B AA^*=A^*$.\\
 (iii) $B A= A^*(AA^*)^{(1)}A.$
\end{cor}

Sufficient conditions for reverse order law of $\{1,4\}$-inverse of
tensors is presented next.

\begin{thm}\label{2.26}
If $\mc{A}^{(1,4)}\n \mc{A}\n \mc{B}\n \mc{B}^*$ is hermitian, then
 $$(\mc{A} \n \mc{B})^{(1,4)}=\mc{B}^{(1,4)} \n \mc{A}^{(1,4)}.$$
\end{thm}

\begin{proof}
The fact $\mc{A}^{(1,4)}\n \mc{A}\n \mc{B}\n \mc{B}^*$ is hermitian
implies
\begin{eqnarray*}
\mc{A}^{(1,4)}\n \mc{A} *_N\mc{B}\n \mc{B}^*
&=& (\mc{A}^{(1,4)}\n \mc{A} *_N\mc{B}\n \mc{B}^*)^*\\
 &=& (\mc{B}\n \mc{B}^*)^* *_N (\mc{A}^{(1,4)} *_N \mc{A})^*\\
 &=& (\mc{B}\n \mc{B}^*) *_N (\mc{A}^{(1,4)} *_N \mc{A}).
\end{eqnarray*}
Hence $\mc{A}^{(1,4)}\n \mc{A}$ and $\mc{B}\n \mc{B}^*$ are
commutative. Let us consider $\mc{X} = \mc{B}^{(1,4)} *_N
\mc{A}^{(1,4)}$ and $\mc{D} = \mc{A} *_N \mc{B}$.
 Then
\begin{eqnarray*}
\mc{D} *_N\mc{X} *_N \mc{D}
&=& \mc{A} *_N\mc{B} *_N \mc{B}^{(1,4)} *_N \mc{A}^{(1,4)} *_N \mc{A} *_N \mc{B}\\
&=& \mc{A} *_N\mc{A}^{(1,4)} *_N \mc{A} *_N \mc{B} *_N\mc{B}^{(1,4)} *_N \mc{B}\\
&=& \mc{A} *_N\mc{B} = \mc{D},
\end{eqnarray*}
since $\mc{A}^{(1,4)} \in \mc{A}\{1\}$ and $\mc{B}^{(1,4)} \in
\mc{B}\{1\}$. Hence  $\mc{X} \in \mc{D}\{1\}$, i.e., $\mc{B}^{(1,4)}
\n \mc{A}^{(1,4)} \in (\mc{A}\n \mc{B})\{1\}.$ Now \begin{eqnarray*}
& &(\mc{B}^{(1,4)} \n \mc{A}^{(1,4)} \n \mc{A}\n
\mc{B})^*\\
  & = &\mc{B}^*\n (\mc{A}^{(1,4)} \n \mc{A})^* \n
(\mc{B}^{(1,4)})^*\\
  & = & \mc{B}^{(1,4)}\n \mc{B} \n \mc{B}^*  \n \mc{A}^{(1,4)} \n \mc{A} \n
(\mc{B}^{(1,4)})^* \\
  & = & \mc{B}^{(1,4)}\n  \mc{A}^{(1,4)} \n \mc{A}\n \mc{B} \n \mc{B}^*  \n
(\mc{B}^{(1,4)})^*,  \\
  & = & \mc{B}^{(1,4)}\n  \mc{A}^{(1,4)} \n \mc{A}\n \mc{B} \n
(\mc{B}^{(1,4)}\n \mc{B})^* \\
  & = & \mc{B}^{(1,4)}\n  \mc{A}^{(1,4)} \n \mc{A}\n \mc{B} \n
\mc{B}^{(1,4)}\n \mc{B} \\
  & = & \mc{B}^{(1,4)}\n  \mc{A}^{(1,4)} \n \mc{A}\n \mc{B}.
  \end{eqnarray*}
  Thus $(\mc{A} \n \mc{B})^{(1,4)}=\mc{B}^{(1,4)} \n \mc{A}^{(1,4)}.$
\end{proof}

The converse of the above theorem is not true, and is shown below
with an example.

\begin{ex}
Let
 $\mc{A} = (a_{ijkl})_{1 \leq i,j,k,l \leq 2}  \in \mathbb{R}^{2\times 2\times 2\times 2}$ and $\mc{B}
 = (b_{ijkl})_{1 \leq i,j,k,l \leq 2}\in \mathbb{R}^{2\times 2\times 2\times 2}$ be two tensors such that
\begin{eqnarray*}
a_{ij11} =
    \begin{pmatrix}
    0 & 0\\
    1 & 0\\
    \end{pmatrix},
a_{ij21} =
    \begin{pmatrix}
    0 & 0\\
    0 & 0\\
    \end{pmatrix},
a_{ij12} =
    \begin{pmatrix}
    0 & -1\\
    0 & 0\\
    \end{pmatrix},
a_{ij22} =
    \begin{pmatrix}
    0 & 1\\
    2 & 0\\
    \end{pmatrix}, ~
\end{eqnarray*}
and
\begin{eqnarray*}
b_{ij11} =
    \begin{pmatrix}
    1 & 0\\
    0 & 1\\
    \end{pmatrix},
b_{ij21} =
    \begin{pmatrix}
    0 & 1\\
    0 & 0\\
    \end{pmatrix},
b_{ij12} =
    \begin{pmatrix}
    0 & -1\\
    0 & 0\\
    \end{pmatrix},
b_{ij22} =
    \begin{pmatrix}
    0 & 0\\
    0 & 1\\
    \end{pmatrix},
\end{eqnarray*}
respectively. Then  $\mc{A}*_N\mc{B} = (c_{ijkl})_{1 \leq i,j,k,l
\leq 2} \in \mathbb{R}^{2\times 2\times 2\times 2}$, $\mc{A}^{(1,4)}
  = (x_{ijkl})_{1 \leq i,j,k,l \leq 2} \in \mathbb{R}^{2\times 2 \times 2\times 2}$, and
$\mc{B}^{(1,4)} = (y_{ijkl})_{1 \leq i,j,k,l \leq 2} \in
\mathbb{R}^{2\times 2 \times 2\times 2}$, where
\begin{eqnarray*}
c_{ij11} =
    \begin{pmatrix}
    0 & 1\\
    3 & 0\\
    \end{pmatrix},
c_{ij21} =
    \begin{pmatrix}
    0 & -1\\
    0 & 0\\
    \end{pmatrix},
c_{ij12} =
    \begin{pmatrix}
    0 & 1\\
    0 & 0\\
    \end{pmatrix},
c_{ij22} =
    \begin{pmatrix}
    0 & 1\\
    2 & 0\\
    \end{pmatrix},
\end{eqnarray*}
\begin{eqnarray*}
x_{ij11} =
    \begin{pmatrix}
    -4 & 1\\
    -1 & 1\\
    \end{pmatrix},
x_{ij21} =
    \begin{pmatrix}
   \frac{1}{3} & \frac{1}{3}\\
    0 & \frac{1}{3}\\
    \end{pmatrix},
x_{ij12} =
    \begin{pmatrix}
    -\frac{1}{3} & -\frac{5}{6}\\
    0 & \frac{1}{6}\\
    \end{pmatrix},
x_{ij22} =
    \begin{pmatrix}
    0 & 1\\
    1 & 1\\
    \end{pmatrix}, ~
\end{eqnarray*}
and
\begin{eqnarray*}
y_{ij11} =
    \begin{pmatrix}
    1 & 0\\
    0 & -1\\
    \end{pmatrix},
y_{ij21} =
    \begin{pmatrix}
    1 & \frac{3}{2}\\
    0 & -\frac{5}{2}\\
    \end{pmatrix},
y_{ij12} =
    \begin{pmatrix}
    0 & -\frac{1}{2}\\
    \frac{1}{2} & 0\\
    \end{pmatrix},
y_{ij22} =
    \begin{pmatrix}
    0 & 0\\
    0 & 1\\
    \end{pmatrix},
\end{eqnarray*}
respectively.  So $ \mc{B}^{(1,4)} *_N \mc{A}^{(1,4)} =
(d_{ijkl})_{1 \leq i,j,k,l \leq 2} \in \mathbb{R}^{2\times 2 \times
2\times 2}$, where
\begin{eqnarray*}
d_{ij11} =
    \begin{pmatrix}
    -5 & -2\\
    \frac{1}{2} & \frac{15}{2}\\
    \end{pmatrix},
d_{ij21} =
    \begin{pmatrix}
    \frac{1}{3} & -\frac{1}{6}\\
    \frac{1}{6} & 0\\
    \end{pmatrix},
d_{ij12} =
    \begin{pmatrix}
    -\frac{1}{3} & \frac{5}{12}\\
    -\frac{5}{12} & \frac{1}{2}\\
    \end{pmatrix},
d_{ij22} =
    \begin{pmatrix}
    1 & 1\\
    \frac{1}{2} & -\frac{3}{2}\\
    \end{pmatrix}.
\end{eqnarray*}
Hence
\begin{eqnarray*}
 (\mc{A}*_N\mc{B})^{(1,4)} = \mc{B}^{(1,4)} *_N \mc{A}^{(1,4)}.
\end{eqnarray*}
However,  $\mc{A}*_N\mc{A}^{(1,4)}*_N\mc{B}^**_N\mc{B}$ is not
hermitian which can be seen from
$\mc{A}*_N\mc{A}^{(1,4)}*_N\mc{B}^**_N\mc{B}= (t_{ijkl})_{1 \leq
i,j,k,l \leq 2} \in \mathbb{R}^{2\times 2 \times 2\times 2}$,  where
\begin{eqnarray*}
t_{ij11} =
    \begin{pmatrix}
    0 & 0\\
    -2 & 0\\
    \end{pmatrix},
t_{ij21} =
    \begin{pmatrix}
    0 & 1\\
    -1 & 0\\
    \end{pmatrix},
t_{ij12} =
    \begin{pmatrix}
    0 & -1\\
    1 & 0\\
    \end{pmatrix},
t_{ij22} =
    \begin{pmatrix}
    0 & 0\\
    0 & 0\\
    \end{pmatrix},
\end{eqnarray*}
and $(\mc{A}*_N\mc{A}^{(1,4)}*_N\mc{B}^**_N\mc{B})^*=
(\overline{t}_{ijkl})_{1 \leq i,j,k,l \leq 2} \in
\mathbb{R}^{2\times 2 \times 2\times 2}$,  where

\begin{eqnarray*}
\overline{t}_{ij11} =
    \begin{pmatrix}
    0 & 0\\
    0 & 0\\
    \end{pmatrix},
\overline{t}_{ij21} =
    \begin{pmatrix}
    0 & 1\\
    -1 & 0\\
    \end{pmatrix},
\overline{t}_{ij12} =
    \begin{pmatrix}
    -2 & -1\\
    1 & 0\\
    \end{pmatrix},
\overline{t}_{ij22} =
    \begin{pmatrix}
    0 & 0\\
    0 & 0\\
    \end{pmatrix}.
\end{eqnarray*}
\end{ex}

As a corollary to Theorem \ref{2.26}, we obtain the following result
for rectangular matrices of suitable order.

\begin{cor}
If $A^{(1,4)}ABB^*$ is hermitian, then
 $$(AB)^{(1,4)}=B^{(1,4)} A^{(1,4)}.$$
\end{cor}

 Similarly, one can have the
following results for $\{1,3\}$-inverse of $\mc{A}.$

\begin{thm}
The following three conditions are equivalent:\\
 (i) $\mc{B} \in \mc{A}\{1,3\}$.\\
 (ii) $ \mc{A}^* \n \mc{A}\n \mc{B} =\mc{A}^*$.\\
 (iii) $\mc{A} \n \mc{B}= \mc{A}\n (\mc{A}^*\n \mc{A})^{(1)}\n
 \mc{A}^*.$
\end{thm}

\begin{thm}\label{2.29}
If $\mc{A}\n \mc{A}^{(1,3)}\n  \mc{B}^*\n \mc{B}$ is hermitian, then
 $$(\mc{A} \n \mc{B})^{(1,3)}=\mc{B}^{(1,3)} \n \mc{A}^{(1,3)}.$$
\end{thm}

The following example shows that the converse of the above result is
not true.

\begin{ex}
Let
 $\mc{A} = (a_{ijkl})_{1 \leq i,j,k,l \leq 2}  \in \mathbb{R}^{2\times 2\times 2\times 2}$ and $\mc{B}
 = (b_{ijkl})_{1 \leq i,j,k,l \leq 2}\in \mathbb{R}^{2\times 2\times 2\times 2}$
 be two tensors such that
\begin{eqnarray*}
a_{ij11} =
    \begin{pmatrix}
    1 & 2\\
    0 & 0\\
    \end{pmatrix},
a_{ij21} =
    \begin{pmatrix}
    1 & 0\\
    0 & 0\\
    \end{pmatrix},
a_{ij12} =
    \begin{pmatrix}
    0 & 0\\
    0 & 1\\
    \end{pmatrix},
a_{ij22} =
    \begin{pmatrix}
    -1 & 0\\
    0 & 0\\
    \end{pmatrix}, ~
\end{eqnarray*}
and
\begin{eqnarray*}
b_{ij11} =
    \begin{pmatrix}
    0 & 0\\
    1 & 0\\
    \end{pmatrix},
b_{ij21} =
    \begin{pmatrix}
    0 & 0\\
    1 & -1\\
    \end{pmatrix},
b_{ij12} =
    \begin{pmatrix}
    1 & 0\\
    0 & 0\\
    \end{pmatrix},
b_{ij22} =
    \begin{pmatrix}
    0 & 0\\
    0 & 1\\
    \end{pmatrix},
\end{eqnarray*}
respectively. Then  $\mc{A}*_N\mc{B}  = (c_{ijkl})_{1 \leq i,j,k,l
\leq 2}  \in \mathbb{R}^{2\times 2\times 2\times 2}$,
 $\mc{A}^{(1,3)}= (x_{ijkl})_{1 \leq i,j,k,l \leq 2} \in
\mathbb{R}^{2\times 2 \times 2\times 2}$, and $\mc{B}^{(1,3)} =
(y_{ijkl})_{1 \leq i,j,k,l \leq 2} \in \mathbb{R}^{2\times 2 \times
2\times 2}$, where
\begin{eqnarray*}
c_{ij11} =
    \begin{pmatrix}
    1 & 0\\
    0 & 0\\
    \end{pmatrix},
c_{ij21} =
    \begin{pmatrix}
    2 & 0\\
    0 & 0\\
    \end{pmatrix},
c_{ij12} =
    \begin{pmatrix}
    1 & 2\\
    0 & 0\\
    \end{pmatrix},
c_{ij22} =
    \begin{pmatrix}
   -1 & 0\\
    0 & 0\\
    \end{pmatrix}, ~
\end{eqnarray*}
\begin{eqnarray*}
x_{ij11} =
    \begin{pmatrix}
    0 & 0\\
    1 & 0\\
    \end{pmatrix},
x_{ij21} =
    \begin{pmatrix}
    0 & 0\\
    1 & 1\\
    \end{pmatrix},
x_{ij12} =
    \begin{pmatrix}
    \frac{1}{2} & 0\\
    1 & \frac{3}{2}\\
    \end{pmatrix},
x_{ij22} =
    \begin{pmatrix}
    0 & 1\\
    1 & 1\\
    \end{pmatrix}, ~
\end{eqnarray*}
and
\begin{eqnarray*}
y_{ij11} =
    \begin{pmatrix}
    -2 & 1\\
    2 & 2\\
    \end{pmatrix},
y_{ij21} =
    \begin{pmatrix}
    0 & 0\\
    1 & 1\\
    \end{pmatrix},
y_{ij12} =
    \begin{pmatrix}
    -1 & 0\\
    1 & 1\\
    \end{pmatrix},
y_{ij22} =
    \begin{pmatrix}
    -1 & 0\\
    1 & 2\\
    \end{pmatrix},
\end{eqnarray*}
respectively.  So $ \mc{B}^{(1,3)} *_N \mc{A}^{(1,3)} =
(d_{ijkl})_{1 \leq i,j,k,l \leq 2} \in \mathbb{R}^{2\times 2 \times
2\times 2}$, where
\begin{eqnarray*}
d_{ij11} =
    \begin{pmatrix}
    0 & 0\\
    1 & 1\\
    \end{pmatrix},
d_{ij21} =
    \begin{pmatrix}
    -1 & 0\\
    2 & 3\\
    \end{pmatrix},
d_{ij12} =
    \begin{pmatrix}
    -\frac{5}{2} & \frac{1}{2}\\
    \frac{7}{2} & 5\\
    \end{pmatrix},
d_{ij22} =
    \begin{pmatrix}
    -2 & 0\\
    3 & 4\\
    \end{pmatrix}.
\end{eqnarray*}
Hence
\begin{eqnarray*}
 (\mc{A}*_N\mc{B})^{(1,3)} = \mc{B}^{(1,3)} *_N \mc{A}^{(1,3)},
\end{eqnarray*}
But $\mc{A}*_N\mc{A}^{(1,3)}*_N\mc{B}^**_N\mc{B}$ is not hermitian
which can be seen from $\mc{A}*_N\mc{A}^{(1,3)}*_N\mc{B}^**_N\mc{B}
 = (t_{ijkl})_{1 \leq i,j,k,l \leq 2} \in
\mathbb{R}^{2\times 2 \times 2\times 2}$, where
\begin{eqnarray*}
t_{ij11} =
    \begin{pmatrix}
    1 & 0\\
    0 & 0\\
    \end{pmatrix},
t_{ij21} =
    \begin{pmatrix}
    0 & 1\\
    0 & 0\\
    \end{pmatrix},
t_{ij12} =
    \begin{pmatrix}
    1 & 0\\
    0 & -1\\
    \end{pmatrix},
t_{ij22} =
    \begin{pmatrix}
    0 & 0\\
    0 & 1\\
    \end{pmatrix},
\end{eqnarray*}
and $(\mc{A}*_N\mc{A}^{(1,3)}*_N\mc{B}^**_N\mc{B})^*
 = (\overline{t}_{ijkl})_{1 \leq i,j,k,l \leq 2} \in
\mathbb{R}^{2\times 2 \times 2\times 2}$, where

\begin{eqnarray*}
\overline{t}_{ij11} =
    \begin{pmatrix}
    1 & 0\\
    1 & 0\\
    \end{pmatrix},
\overline{t}_{ij21} =
    \begin{pmatrix}
    0 & 1\\
    0 & 0\\
    \end{pmatrix},
\overline{t}_{ij12} =
    \begin{pmatrix}
    0 & 0\\
    0 & 0\\
    \end{pmatrix},
\overline{t}_{ij22} =
    \begin{pmatrix}
    0 & 0\\
    -1 & 1\\
    \end{pmatrix}.
\end{eqnarray*}
\end{ex}

%
%
%
%

We now present one of our important results, which yields a method
of construction of the Moore-Penrose inverse of a tensor using $\{
1,3\}$-inverse and $\{ 1,4\}$-inverse of $\mc{A}$. One can find the
matrix version of these results  in \cite{BenGr03}.

\begin{thm}
Let $\mc{A}\in \mathbb{C}^{I_1\times\cdots\times I_N \times J_1
\times\cdots\times J_N }$. Then
\begin{equation*}
\mc{A}^{\dg} =\mc{A}^{(1,4)}*_N\mc{A}*_N\mc{A}^{(1,3)}.
\end{equation*}
\end{thm}

\begin{proof}
Suppose that $\mc{X}=\mc{A}^{(1,4)}*_N\mc{A}*_N\mc{A}^{(1,3)}.$ Then
\begin{eqnarray*}
\mc{A}*_N\mc{X}*_N\mc{A}
&=&\mc{A}*_N(\mc{A}^{(1,4)}*_N\mc{A}*_N\mc{A}^{(1,3)})*_N\mc{A}=
\mc{A}*_N\mc{A}^{(1,3)}*_N\mc{A}=\mc{A},\\
\mc{X}*_N\mc{A}*_N\mc{X} &=&(\mc{A}^{(1,4)}*_N\mc{A}*_N\mc{A}^{(1,3)})
*_N\mc{A}*_N(\mc{A}^{(1,4)}*_N\mc{A}*_N\mc{A}^{(1,3)})\\
&=&\mc{A}^{(1,4)}*_N(\mc{A}*_N\mc{A}^{(1,3)}
*_N\mc{A})*_N\mc{A}^{(1,4)}*_N\mc{A}*_N\mc{A}^{(1,3)}\\
&=&\mc{A}^{(1,4)}*_N\mc{A}*_N\mc{A}^{(1,4)}*_N\mc{A}
*_N\mc{A}^{(1,3)}\\
&=&\mc{A}^{(1,4)}*_N(\mc{A}*_N\mc{A}^{(1,4)}*_N\mc{A})
*_N\mc{A}^{(1,3)}\\
&=&\mc{A}^{(1,4)}*_N\mc{A}*_N\mc{A}^{(1,3)}=\mc{X},\\
 (\mc{A}\n
\mc{X})^*
&=& \left(\mc{A}*_N(\mc{A}^{(1,4)}*_N\mc{A}*_N\mc{A}^{(1,3)})\right)^*\\
&=& \left((\mc{A}*_N\mc{A}^{(1,4)}*_N\mc{A})*_N\mc{A}^{(1,3)}\right)^*\\
&=& \left(\mc{A}*_N\mc{A}^{(1,3)}\right)^*\\
&=& \mc{A}*_N\mc{A}^{(1,3)}\\
&=& \mc{A}*_N \mc{A}^{(1,4)}*_N\mc{A}*_N\mc{A}^{(1,3)}\\
&=& \mc{A} \n \mc{X}, \\
\end{eqnarray*}
and
\begin{eqnarray*}
(\mc{X}\n \mc{A})^*
&=& \left((\mc{A}^{(1,4)}*_N\mc{A}*_N\mc{A}^{(1,3)})*_N\mc{A}\right)^*\\
&=& \left(\mc{A}^{(1,4)}*_N(\mc{A}*_N\mc{A}^{(1,3)}*_N\mc{A}\right)^*\\
&=& \left(\mc{A}^{(1,4)}*_N\mc{A}\right)^*\\
&=& \mc{A}^{(1,4)}*_N(\mc{A}*_N\mc{A}^{(1,3)}*_N\mc{A})\\
&=& (\mc{A}^{(1,4)}*_N\mc{A}*_N\mc{A}^{(1,3)})*_N\mc{A} = \mc{X}\n
\mc{A}.
\end{eqnarray*}
Hence $\mc{X} = \mc{A}^{\dg}.$
\end{proof}

We remark that the first two conditions also follow from Lemma
\ref{12}. The Moore-Penrose inverse of Kronecker product of two
tensors can be computed using the next result.

\begin{thm}
Let $\mc{A} \in \mathbb{C}^{I_1\times\cdots\times I_N \times J_1
\times\cdots\times J_N} $ and $\mc{B} \in
\mathbb{C}^{K_1\times\cdots\times K_M \times L_1 \times\cdots\times
L_M} $. Then $$(\mc{A}\otimes\mc{B})^{\dg} =
\mc{A}^{\dg}\otimes\mc{B}^{\dg}.$$
\end{thm}

\begin{proof}
Suppose that $\mc{K} = \mc{A}\otimes\mc{B}$ and $\mc{X} =
\mc{A}^{\dg}\otimes\mc{B}^{\dg}$. We now have $\mc{K}*_N\mc{X} =
(\mc{A}\otimes\mc{B})*_N(\mc{A}^{\dg}\otimes\mc{B}^{\dg}) =
\mc{A}*_N\mc{A}^{\dg} \otimes\mc{B}*_N\mc{B}^{\dg}$ by Lemma
\ref{kp}. So
$\mc{K}*_N\mc{X}*_N\mc{K}=\mc{A}*_N\mc{A}^{\dg}*_N\mc{A}*_N \otimes
\mc{B}*_N\mc{B}^{\dg}*_N\mc{B} = \mc{A}\otimes\mc{B}=\mc{K}$, and
$\mc{K}*_N\mc{X}$ is hermitian by Proposition \ref{Krop} (a).
Similarly, the other two conditions can be shown.
\end{proof}

\section{Multilinear system}
Sylvester  matrix equation plays significant roles in system and
control theory \cite{CheLu12, QiuCh04, LizBaCY15}. One can  compute
exact solution of such an equation by using the Kronecker product,
but the computational efforts rapidly increase with the dimensions
of the matrices to be solved \cite{CheLu12}. The Sylvester tensor
 equation  via the Einstein product can be written in the following way:
\begin{equation}
\mc{A}*_N\mc{X} + \mc{X}*_M\mc{B} = \mc{D},
\end{equation}
where $ \mc{A} \in \mathbb{R}^{I_1\times\cdots\times I_N \times I_1
\times\cdots\times I_N}, ~\mc{X} \in
\mathbb{R}^{I_1\times\cdots\times I_N \times J_1 \times\cdots\times
J_M}, ~\mc{B} \in \mathbb{R}^{J_1\times\cdots\times J_M \times J_1
\times\cdots\times J_M}$, and $ \mc{D} \in
\mathbb{R}^{I_1\times\cdots\times I_N \times J_1 \times\cdots\times
J_M}$. This appears in the finite element method \cite{Gra04},
finite difference or spectral method \cite{LiSu09, LiTiSM10}, and
plays an important role in discretization of a linear partial
differential equation in high dimension. Further, based on the
operations of `block tensors', one can write block tensor equation
as:

\begin{equation}
\left[
\begin{array}{cc} \mc{A} & \mc{I}_1
\end{array}
\right] *_N
\begin{bmatrix}
 \mc{X} & \mc{O} \\ \mc{O} & \mc{X}
\end{bmatrix} *_N
\left[
\begin{array}{c} \mc{I}_2 \\ \mc{B}
\end{array}
\right] = \mc{D},
\end{equation}
where $\mc{I}_1 = \mathbb{R}^{I_1\times\cdots\times I_N \times I_1
\times\cdots\times I_N}$ and $\mc{I}_2 =
\mathbb{R}^{J_1\times\cdots\times J_M \times J_1 \times\cdots\times
J_M}$ are unit tensors. Equivalently, we have

\begin{equation}\label{system41}
\mc{A}*_N\mc{X}*_M\mc{B} = \mc{D},
\end{equation}
where $ \mc{A} = \mathbb{R}^{I_1\times\cdots\times I_N \times J_1
\times\cdots\times J_N}, $ $ \mc{X} =
\mathbb{R}^{J_1\times\cdots\times J_N \times K_1 \times\cdots\times
K_M}, $ $ \mc{B} = \mathbb{R}^{K_1\times\cdots\times K_M \times L_1
\times\cdots\times L_M} $ and $ \mc{D} =
\mathbb{R}^{I_1\times\cdots\times I_N \times L_1 \times\cdots\times
L_M}$.

Sun {\it et al.}, \cite{LizBaCY15}  proved the following result for
solving equation (\ref{system41}).

\begin{thm} \label{samethm}{(Theorem 4.1, \cite{LizBaCY15})}\\
Let $\mc{A} \in \mathbb{R}^{I_1\times\cdots\times I_N \times J_1
\times\cdots\times J_N},
    \mc{X} \in \mathbb{R}^{J_1\times\cdots\times J_N \times K_1 \times\cdots\times K_M},
    \mc{B}  \in \mathbb{R}^{K_1\times\cdots\times K_M \times L_1 \times\cdots\times L_M}$ and
   $\mc{D} \in \mathbb{R}^{I_1\times\cdots\times I_N \times L_1 \times\cdots\times L_M}$. Then the tensor equation
$\mc{A}*_N\mc{X}*_M\mc{B} = \mc{D}$;
\begin{itemize}
\item[(a)] is solvable if and only if  there exist $\mc{A}^{(1)}$ and $\mc{B}^{(1)}$
such that $$\mc{A}*_N\mc{A}^{(1)}*_N\mc{D}*_M\mc{B}^{(1)}*_M\mc{B} =
\mc{D},$$
\item[(b)] in this case, the general solution is
\begin{equation}
\mc{X} = \mc{A}^{(1)}*_N\mc{D}*_M\mc{B}^{(1)} + \mc{Z} -
\mc{A}^{(1)}*_N\mc{A}*_N\mc{Z}*_M\mc{B}*_M\mc{B}^{(1)},
\end{equation}
where $\mc{Z}\in \mathbb{R}^{{J_1\times\cdots\times J_N\times
K_1\times\cdots\times K_M}}$ is an arbitrary tensor.
\end{itemize}
\end{thm}

The next result is obtained as a corollary to the above one.

\begin{cor}
Let $\mc{A} \in \mathbb{R}^{I_1\times\cdots\times I_M \times J_1
\times\cdots\times J_N}$ and
 $\mc{A}^{(1)} \in \mc{A}{\{ 1\}}.$ Then
$$
\mc{A}{\{ 1\}} = \{\mc{A}^{(1)} + \mc{Y} -
\mc{A}^{(1)}*_M\mc{A}*_N\mc{Y}*_M\mc{A}*_N\mc{A}^{(1)}:
 \mc{Y} \in \mathbb{R}^{J_1\times\cdots\times J_N \times I_1 \times\cdots\times I_M}\}.
$$
\end{cor}

\begin{proof}
By Theorem \ref{samethm}, the general solution of
$\mc{A}*_N\mc{X}*_M\mc{A} = \mc{A}$ is:
$$
\mc{X} = \mc{A}^{(1)}*_M\mc{A}*_N\mc{A}^{(1)} + \mc{Z} -
\mc{A}^{(1)}*_M\mc{A}*_N\mc{Z}*_M\mc{A}*_N\mc{A}^{(1)},
$$ where $\mc{Z}$ is arbitrary.
Substituting $\mc{Z} = \mc{A}^{(1)}+\mc{Y}$, we get
\begin{eqnarray*}
\mc{X} &=& \mc{A}^{(1)}*_M\mc{A}*_N\mc{A}^{(1)} +
\mc{A}^{(1)}+\mc{Y} - \mc{A}^{(1)}*_M\mc{A}*_N\mc{A}^{(1)}*_M\mc{A}*_N\mc{A}^{(1)}\\
&& ~~~~~~~~~~~~~~~~~~~~~~~~~~~~~~~~~~~~~~~~~~~ -\mc{A}^{(1)}*_M\mc{A}*_N\mc{Y}*_M\mc{A}*_N\mc{A}^{(1)}\\
&=& \mc{A}^{(1)}*_M\mc{A}*_N\mc{A}^{(1)} + \mc{A}^{(1)}+\mc{Y}- \mc{A}^{(1)}*_M\mc{A}*_N\mc{A}^{(1)}\\
&& ~~~~~~~~~~~~~~~~~~~~~~~~~~~~~~~~~~~~~~~~~~
-\mc{A}^{(1)}*_M\mc{A}*_N\mc{Y}*_M\mc{A}*_N\mc{A}^{(1)}\\
&=&
\mc{A}^{(1)}+\mc{Y}-\mc{A}^{(1)}*_M\mc{A}*_N\mc{Y}*_M\mc{A}*_N\mc{A}^{(1)}.
\end{eqnarray*}
Hence $$ \mc{A}{\{ 1\}} = \{\mc{A}^{(1)} + \mc{Y} -
\mc{A}^{(1)}*_M\mc{A}*_N\mc{Y}*_M\mc{A}*_N\mc{A}^{(1)}:
 \mc{Y} \in \mathbb{R}^{J_1\times\cdots\times J_N \times I_1 \times\cdots\times I_M}\}.
$$
\end{proof}
The result produced hereunder is a special case of Theorem
\ref{samethm} in the setting of system of linear equations of
tensors.

\begin{cor}\label{axbcor}
Let $\mc{A} \in \mathbb{R}^{I_1\times\cdots\times I_N \times J_1
\times\cdots\times J_N} ~\mbox{and}~~ \mc{B}
 \in \mathbb{R}^{I_1\times\cdots\times I_N}$, then the equation
\begin{equation}
\label{five5} \mc{A}*_N \mc{X} = \mc{B}
\end{equation}
is consistent if and only if  for some $\mc{A}^{(1)}$ such that
\begin{equation}
\label{six6} \mc{A}*_N\mc{A}^{(1)}*_N \mc{B} = \mc{B}
\end{equation}
in which case the general solution of equation \eqref{five5} is:
\begin{equation}
\mc{X} = \mc{A}^{(1)}*_N \mc{B}+(\mc{I}-\mc{A}^{(1)}*_N\mc{A})*_N\mc{Y},
\end{equation}
for any arbitrary $\mc{Y}\in \mathbb{R}^{J_1 \times\cdots\times
J_N}.$
\end{cor}

An alternative proof of Theorem \ref{samethm} is provided below
using Kronecker product of tensors.

\begin{proof}
The tensor equation $ \mc{A} \n \mc{X}*_{M} \mc{B}=\mc{D}$  can be
rewritten as $(\mc{A}\otimes\mc{B}^*)*_{(N+M)} \mc{X}=\mc{D}$, where
$\mc{A}\otimes\mc{B}^*$ is Kronecker product as defined in Section
2. Applying Corollary \ref{axbcor} to
$(\mc{A}\otimes\mc{B}^*)*_{(N+M)} \mc{X}=\mc{D}$, we have the
general solution of the form:
\begin{equation*}
\mc{X} = (\mc{A} \otimes \mc{B}^*)^{(1)} *_{(N+M)} \mc{D} + [\mc{I}- (\mc{A}
 \otimes \mc{B}^*)^{(1)}*_{M} (\mc{A} \otimes \mc{B}^*)]*_{(N+M)}  \mc{Z},
\end{equation*}
 where $\mc{Z}$ is arbitrary. Since we have
$(\mc{A} \otimes \mc{B}^*)^{(1)}= \mc{A}^{(1)} \otimes{(\mc{B}^*)}^{(1)}$. Then, we get
$(\mc{A} \otimes \mc{B}^*)^{(1)} *_{M}(\mc{A} \otimes \mc{B}^*) = (\mc{A}^{(1)}\n \mc{A})
\otimes((\mc{B}^*)^{(1)}*_M\mc{B}^{*}).$
Hence
$$
\mc{X} = (\mc{A} \otimes \mc{B}^*)^{(1)}*_{(N+M)} \mc{D} + [\mc{I}- (\mc{A}^{(1)}*_N \mc{A})
 \otimes ((\mc{B}^*)^{(1)}*_M \mc{B}^*)]*_{(N+M)} \mc{Z}.
$$
Thus, we arrive at the general solution of the form
\begin{equation*}
\mc{X} = \mc{A}^{(1)}*_N\mc{D}*_M\mc{B}^{(1)} + \mc{Z} -
\mc{A}^{(1)}*_N\mc{A}*_N\mc{Z}*_M\mc{B}*_M\mc{B}^{(1)}.
\end{equation*}
\end{proof}

The idea of the above proof is borrowed from the book \cite{Rao71}
where the authors proved for matrices. Next two results are about
solution of two and three tensor equations.

\begin{thm}
The tensor equations $\mc{A}*_N\mc{X} = \mc{B}$ and $\mc{X}*_N\mc{D}
= \mc{F}$ has a common solution if and only if  each equation
separately has a solution and $\mc{A}*_N\mc{F} = \mc{B}*_N\mc{D}.$
\end{thm}

\begin{proof}
Let $ \mc{X} = \mc{A}^{(1)}*_N\mc{B} + \mc{F}*_N\mc{D}^{(1)} -
\mc{A}^{(1)}*_N\mc{A}*_N\mc{F}*_N\mc{D}^{(1)}$. Then $\mc{A}\n
\mc{X}=\mc{A}*_N\mc{A}^{(1)}\n \mc{B}$ and $\mc{X}\n \mc{D}=
\mc{A}^{(1)}\n \mc{B} \n \mc{D}+ \mc{F}\n \mc{D}^{(1)} \n \mc{D} -
\mc{A}^{(1)}*_N\mc{A}*_N\mc{F}*_N\mc{D}^{(1)} \n \mc{D}$. So
$\mc{X}$ is a common solution of both equations $\mc{A}*_N\mc{X} =
\mc{B}$ and $\mc{X}*_N\mc{D} = \mc{F}$, provided $\mc{A}*_N\mc{F} =
\mc{B}*_N\mc{D}$, $\mc{A}*_N\mc{A}^{(1)}\n \mc{B} = \mc{B}$ and
$\mc{F}*_N\mc{D}^{(1)} \n \mc{D}
 = \mc{F}.$ However, the last  two tensor equations are equivalent
 to the consistency condition of Theorem \ref{samethm} for the tensor equations
 $\mc{A}*_N\mc{X} = \mc{B}$ and $\mc{X}*_N\mc{D}
= \mc{F}$.  The other way is obvious.
\end{proof}

\begin{thm}
There is at most one tensor $\mc{X}$ satisfying these three
relations
 $$\mc{A}*_N\mc{X} = \mc{B}, ~~\mc{X}*_N\mc{A} = \mc{D}~\mbox{and}~\mc{X}*_N\mc{A}*_N\mc{X} =
 \mc{X}.$$
\end{thm}

\begin{proof}
Suppose that there exists  another  tensor $\mc{Y}$  satisfying
these properties. Then
\begin{eqnarray*}
\mc{X} &=& \mc{X}*_N\mc{A}*_N\mc{X}=\mc{X}*_N\mc{B} = \mc{X}*_N\mc{A}*_N\mc{Y}\\
& =& \mc{D}*_N\mc{Y} = \mc{Y}*_N\mc{A}*_N\mc{Y} = \mc{Y}.
\end{eqnarray*}
\end{proof}

Applying  Theorem \ref{samethm} to $\mc{A}\n \mc{X}= \mc{A}\n
\mc{A}^{(1,3)}$ and putting $\mc{Z}= \mc{Y}+\mc{A}^{(1,3)} $, we
have the following corollary which gives a characterization of class
of $\{ 1,3\}$-inverse of $\mc{A}$.

\begin{cor}
Let $\mc{A} \in \mathbb{C}^{I_1\times\cdots\times I_M \times J_1
\times\cdots\times J_N}$ and $ \mc{A}^{(1,3)}\in  \mc{A}{\{ 1,3\}}.$
Then
\begin{equation*}
\mc{A}{\{ 1,3\}} = \{ \mc{A}^{(1,3)} + (\mc{I}-
\mc{A}^{(1,3)}*_M\mc{A})*_N\mc{Y} :
 \mc{Y} \in \mathbb{C}^{J_1 \times\cdots\times J_N \times I_1\times\cdots\times I_M} \}.
\end{equation*}
\end{cor}

Similar consideration of tensor system $\mc{X}*_M \mc{A}=
\mc{A}^{(1,4)}*_M \mc{A}$ leads to the next corollary resulting a
characterization of class of $\{ 1,4\}$-inverse of $\mc{A}$.

\begin{cor}
Let $\mc{A} \in \mathbb{C}^{I_1\times\cdots\times I_M \times J_1
\times\cdots\times J_N}$ and $ \mc{A}^{(1,4)}\in  \mc{A}{\{ 1,4\}}.$
Then
\begin{equation}
\mc{A}{\{ 1,4\}} = \{ \mc{A}^{(1,4)} + \mc{Y}*_M(\mc{I}-
\mc{A}*_N\mc{A}^{(1,4)}) : \mc{Y} \in
\mathbb{C}^{J_1 \times\cdots\times J_N \times I_1\times\cdots\times I_M} \}.
\end{equation}
\end{cor}

We thus conclude this section with an analogous result of Corollary
\ref{axbcor} where $\{ 1\}$-inverse is replaced by the Moore-Penrose
inverse.

\begin{thm}
$\mc{A}*_N \mc{X}=\mc{B}$ has a solution if and only if $\mc{A}*_N
\mc{A}^{\dg} *_N  \mc{B} = \mc{B}$. If a solution exists,
 then every solution is of the form
\begin{equation}
\mc{X}=\mc{A}^\dg *_N  \mc{B} + (\mc{I}- \mc{A}^\dg \n \mc{A})*_N
\mc{W},
\end{equation}
where $\mc{W}$ is  arbitrary.
\end{thm}
\begin{proof}
First, we  verify the consistency condition. Clearly, if $\mc{A} *_N
\mc{A}^{\dg} *_N   \mc{B} = \mc{B}$ then $\mc{X}_0 = \mc{A}^\dg *_N  \mc{B} $
is a solution. Conversely, suppose that a solution exists. Then
$\mc{A}*_N  \mc{Y}=\mc{B}$ yields
 $\mc{A}*_N \mc{A}^\dg*_N\mc{A}*_N \mc{Y}  =  \mc{A}*_N\mc{A}^\dg *_N  \mc{B}$
which implies $ \mc{B}= \mc{A}*_N  \mc{Y}  =  \mc{A}*_N\mc{A}^\dg
*_N  \mc{B}$.

Now suppose that the system has a solution $\mc{X}_0 = \mc{A}^\dg
*_N  \mc{B}.$ Then, for $\mc{W} = \mc{X}-\mc{X}_0$, we have
$\mc{A}*_N \mc{W}= \mc{A} *_N  \mc{X}- \mc{A} *_N  \mc{X}_0 =
\mc{B}- \mc{A}*_N \mc{A}^\dg *_N  \mc{B} = \mc{B}-\mc{B} = \mc{O}$.
But $\mc{A} *_N \mc{W} = \mc{O}$ implies $ \mc{A}^\dg *_N\mc{A}*_N
\mc{W} = \mc{O} $. Hence $\mc{W} = \mc{W} - \mc{A}^\dg *_N \mc{A}*_N
\mc{W} = (\mc{I}- \mc{A}^\dg *_N \mc{A})*_N  \mc{W}$. Thus
$\mc{X}=\mc{X}_0 + \mc{W} = \mc{A}^\dg *_N  \mc{B} + (\mc{I}-
\mc{A}^\dg *_N \mc{A})*_N \mc{W}$.
\end{proof}

\section{Conclusion}

We have further added some results on generalized inverses  of
tensors via the Einstein product to the existing theory.
 The matrix analogue of  many results presented in this paper are
available in the famous book \cite{BenGr03} and \cite{Rao71}. During
the discussion, we encountered the following issues which have not
been addressed in this paper,
 and are left as open problems for future
studies.
 Looking at Remark 2 of Section 2, one may ask the question given
 next.

 {\bf Question 1.} When does  $(\mc{A}*_N\mc{B})^\dg = \mc{B}^\dg *_N
 \mc{A}^\dg $ ?

The above one can also be stated as reverse order law for
Moore-Penrose inverse of tensors via the Einstein product. The other
problem is noted below.

 {\bf Question
2.} Unlike matrices, does there exist a full rank factorization of
tensors ? If so, can this be used to compute the Moore-Penrose
inverse of a tensor $\mc{A}$ ?

\vspace{1cm}

 \noindent {\small {\bf Acknowledgments.}

 The second author also
acknowledges the support provided by Science and Engineering
Research Board, Department of Science and Technology, New Delhi,
India, under the grant number YSS/2015/000303.}

\bibliographystyle{plain}
\bibliography{my_tensor_refs}
\end{document}